\documentclass{amsart}

\usepackage{amssymb,latexsym,amscd,amsthm,amsxtra,amsmath}
\usepackage[all]{xy}
\usepackage[mathscr]{eucal}
\usepackage{stmaryrd}
\usepackage{fullpage}

\usepackage{graphicx}
\usepackage{color}

\theoremstyle{plain}
\newtheorem{lemma}{Lemma}[section]
\newtheorem{proposition}[lemma]{Proposition}
\newtheorem{theorem}[lemma]{Theorem}
\newtheorem{corollary}[lemma]{Corollary}

\theoremstyle{definition}
\newtheorem{definition}[lemma]{Definition}

\newtheorem{example}[lemma]{Example}

\newtheorem{remark}[lemma]{Remark}

\newtheorem{question}[lemma]{Question}

\theoremstyle{remark}

\numberwithin{equation}{section}

\newcommand{\lcm}{\operatorname{lcm}}

\newcommand{\id}{\operatorname{id}}

\newcommand{\Link}{\operatorname{Link}}
\newcommand{\bm}{\operatorname{{\bf bm}}}

\newcommand{\letbe}{\mathbin{\raisebox{.3pt}{:}\!=}}

\begin{document}

\title{Bistellar Cluster Algebras and Piecewise Linear Invariants}

\author{Alastair Darby}
\address{Department of Pure Mathematics\\
Xi'an Jiaotong--Liverpool University\\
Suzhou\\
215123\\
P.\ R.\ China.}
\email{Alastair.Darby@xjtlu.edu.cn}

\author{Fang Li}
\address{Department of Mathematics\\
Zhejiang University\\
Hangzhou\\
310027\\
P.\ R.\ China.}
\email{fangli@zju.edu.cn}

\author{Zhi L\"u}
\address{School of Mathematical Sciences\\
Fudan University\\
Shanghai\\
200433\\
P.\ R.\ China.}
\email{zlu@fudan.edu.cn}

\date{\today}


\begin{abstract}
Inspired by the ideas and techniques used in the study of cluster algebras we construct a new class of algebras, called \emph{bistellar cluster algebras}, from closed oriented triangulated even-dimensional manifolds by performing middle-dimensional bistellar moves. This class of algebras exhibit the algebraic behaviour of middle-dimensional bistellar moves but do not satisfy the classical cluster algebra axiom: ``every cluster variable in every cluster is exchangeable''.    Thus the construction of bistellar cluster algebras is quite different from that of a classical cluster algebra.


 Secondly, using bistellar cluster algebras and the techniques of combinatorial topology, we construct a direct system associated with a set of PL homeomorphic  PL manifolds of dimension 2 or 4, and show that the limit of this direct system is a PL invariant.
\end{abstract}

\maketitle

\section{Introduction}

Finding invariants of manifolds is important in solving problems such as the \emph{smooth 4- dimensional Poincar\'e conjecture} (see~\cite{fgmw}) which asks if there are different smooth structures on the 4-sphere. In this paper we describe a new invariant of Piecewise Linear manifolds, which in the case of 4~dimensional manifolds (and lower), are the same as smooth manifolds. This has the potential of helping to answer the smooth Poincar\'e conjecture if different candidates for the smooth/PL structures can be shown to have distinct invariants. The invariants we have found for 2- and 4-dimensional smooth/PL manifolds are inspired by the work in the field of cluster algebras along with classical work in PL manifolds as described below.

\subsection{Background and motivation}

Cluster algebras were first introduced and studied systematically in a series of papers ~\cite{fz1,fz2, bfz, fz3}.  The original motivation and aim for cluster algebras was to provide an algebraic framework for the study of total positivity phenomena in semisimple Lie groups and of Lusztig's and Kashiwara's  canonical bases of quantum groups. In recent years, it has been shown that this kind of algebraic structure for cluster algebras can arise in various mathematical settings, such as total positivity, Lie theory, quiver representations, Teichm\"uller theory, Poisson geometry, discrete dynamical systems, tropical geometry, algebraic combinatorics, topology of surfaces and so on (see, for example,~\cite{ck1, ck2, fz4}).

Roughly speaking, a cluster algebra $\mathcal{A}$ of rank $r$ is a subalgebra of an ambient field $\mathcal{F}$ of rational functions in $r$ indeterminates, which is defined via a distinguished
family of generators. These generators are called {\em cluster variables} and are grouped into $r$-subsets called {\em clusters}. These clusters satisfy the  exchange property: for any cluster ${\bf x}$ and any element
$x \in{\bf  x}$, there is another cluster ${\bf x}'$ and $x'\in {\bf x}'$ with ${\bf x}'=({\bf x}\setminus\{x\})\cup \{x'\}$ such that  $x$ and  $x'$ are
related  by an {\em exchange relation}
\[
xx' \,=\, M^+ +M^-,
\]
where $M^+$ and $M^-$ are two monomials without common divisors in the $r-1$ variables
${\bf x}\cap {\bf x}'$, and the exponents in $M^+$ and $M^-$ are encoded in an $(r\times r)$ integral
matrix $B = (b_{ij})$ (usually skew-symmetric) called the {\em exchange
matrix}. The cluster algebra $\mathcal{A}$ can also be obtained from  a pair $({\bf x}, B)$, called a seed, by seed mutations, where ${\bf x}$ is an initial cluster and $B$ is an initial exchange matrix, so that all cluster variables can be  constructed
inductively, and each cluster variable $x$ is a rational function  of the variables in the initial cluster ${\bf x}$. With this point of view, the cluster algebra $\mathcal{A}$
becomes a subring of the field formed by all rational functions  of the variables in the initial cluster ${\bf x}$.

In their work~\cite{fst}, Fomin, Shapiro and Thurston showed that any triangulation $T$ of a connected oriented Riemann surface $S$ produces a cluster algebra $\mathcal{A}_T$, which is briefly stated  as follows: they regard  the set of all edges (or $1$-dimensional faces) of $T$  as an initial cluster ${\bf x}(T)$
and  define a skew-symmetric
matrix $B(T)$  by looking at the (signed) adjacencies between the edges of $T$, which  serves as an initial exchange matrix of the cluster algebra $\mathcal{A}_T$.
Then the required cluster algebra $\mathcal{A}_T$ is obtained from the initial seed $({\bf x}(T), B(T))$ by seed mutations. In particular, they showed that  mutations of matrices in algebra completely agree with 1-moves (or flips) in combinatorial topology. It is well-known that  any connected oriented Riemann surface  $S$ has a unique PL structure and
any two triangulations with the same number of vertices of $S$ are connected by 1-moves. Thus, for any two triangulations $T$ and $T'$ with the same number of vertices of $S$,
both $\mathcal{A}_T$ and $\mathcal{A}_{T'}$ are isomorphic. This means that this class of cluster algebras are invariant up to isomorphism under the actions of 1-moves.

The following question naturally arises:

\begin{question}\label{q1}
What kinds of  analogues of cluster algebras can be produced from high-dimensional oriented triangulated manifolds?
\end{question}

Although the 2-dimensional case provides us with intuition on the study of the higher dimensional cases, we note that the problem is more complicated than it may seem at first glance. This heavily depends upon the combinatorial topology of the various moves in the higher dimensional cases.

Actually, Question~\ref{q1} is  directly related to the problem of how to give an algebraic description of various moves (especially bistellar moves) in the higher dimensional cases.  Pachner~\cite{pa1, pa2} tells us that
two PL manifolds  are PL homeomorphic if and only if one can be transformed into the other by a finite sequence of bistellar moves. We therefore pay more attention to the case of bistellar moves in general.

\subsection{Our work and statements of main results}

Let $K$ be a connected oriented triangulated closed manifold of dimension $n\geq 2$. To extend the theory of Fomin, Shapiro and Thurston in the 2-dimensional case to the general case,
we need to consider the following questions:
\begin{enumerate}
\item What will our {\em cluster variables} will be?
\item How should we define the {\em exchange matrices}?
\item How should the {\em matrix mutations} which correspond to
the bistellar moves be defined?
\item How can we define the {\em exchange relations} among seeds?
\end{enumerate}
These questions are closely related to each other. We first investigate how the $f$-vector of $K$ changes under the action of
bistellar moves. Together with a theorem of Pachner~\cite{pa1}, we show that the $f$-vector of $K$ does not change when $n$ is even and we only perform middle-dimensional bistellar moves (see Proposition~\ref{f-vector}). Based upon this, we
shall pay more attention to the case which $n=2h$ is even, and  study the algebraic behaviour of performing middle-dimensional bistellar moves on $K$.

The definition of the exchange matrix of $K$ depends upon the choice of cluster variables. On the other hand, in order to make sure that the definition of the exchange matrix of $K$ is well-defined,
we need to fix the orientations of all $n$-simplices (or lower-dimensional simplices)  of $K$, as shown in \cite[Definition 4.1]{fst} for the 2-dimensional case.
 Indeed, we can use an orientation class of $H_n(K)$  to fix the orientations of all $n$-simplices. However, there is no canonical way to fix orientations of simplices of dimension less than $n$.  This leads us to  choose all $(n-1)$-simplices of $K$ as cluster variables. We give a generalized boundary operator for oriented simplices, and then use it to define the pair orderings among all $(n-1)$-simplices of an oriented $n$-simplex. When $n=2$, Fomin, Shapiro and Thurston use clockwise and counter-clockwise more visually to define such pair orderings among three edges of an oriented triangle, as shown in \cite[Definition 4.1]{fst}.  Furthermore, we can define the exchange matrix $B(K)$ of $K$. Our definition is a direct generalization of  \cite[Definition 4.1]{fst}. It should be pointed out that our definition can still be carried out even if $n$ is odd.

As we will see,  if $n=2h$, when we perform an $h$-dimensional bistellar move  on $K$, then we will change $\binom{h+1}{2}$ cluster
variables (that is, $(n-1)$-simplices of $K$) into $\binom{h+1}{2}$ new cluster variables. When $n=2$, $\binom{h+1}{2}=1$ and so we replace a cluster variable by a new one. This is the case of a classical cluster algebra. However, if $n>2$, we need to replace more cluster variables  when we perform an $h$-dimensional bistellar move  on $K$. This results in an essential difference from the 2-dimensional case.

To define the matrix mutation such that it agrees with the corresponding middle-dimensional bistellar move, we analyze the combinatorial structure of an $h$-dimensional bistellar move ${\bf bm}_\alpha$ on $K$. This associates to a bistellar pair $(\alpha, \beta)$, and two sub-complexes $\Lambda_\alpha$ and $\Lambda_\beta$ of $K$ and ${\bf bm}_\alpha K$ respectively (see Defintion~\ref{def-bistellar} for the meaning of ${\bf bm}_\alpha K$). When we perform an $h$-dimensional bistellar move ${\bf bm}_\alpha$ on $K$, actually the resulting manifold ${\bf bm}_\alpha K$ is obtained from $K$ by replacing $\Lambda_\alpha$ with $\Lambda_\beta$, and in particular, only $\binom{h+1}{2}$ cluster variables (that is, $(n-1)$-simplices) in $\Lambda_\alpha$ will be replaced by $\binom{h+1}{2}$ cluster variables in $\Lambda_\beta$. Note that both $\Lambda_\alpha$ and $\Lambda_\beta$ have $\binom{h+1}{1}\binom{h+1}{1}+\binom{h+1}{2}$ simplices of dimension $n-1$ (that is, cluster variables). Then we can choose a canonical combinatorial equivalence $\sigma$ between $\Lambda_\alpha$ and $\Lambda_\beta$, which can be regarded as a permutation of order 2, and
sends the $\binom{h+1}{2}$ cluster variables  in $\Lambda_\alpha$ into the $\binom{h+1}{2}$ cluster variables in $\Lambda_\beta$. Moreover,
we define the required  matrix mutation
\[
\mu_{\alpha}\colon B(K)\longrightarrow B({\bf bm}_\alpha K)
\]
along \emph{multi-directions} (see Definition~\ref{mm-def}), so that  the matrix mutation $\mu_{\alpha}$ agrees with the bistellar move ${\bf bm}_\alpha$ for exchange matrices, stated as follows:\\
\\
{\bf Proposition~\ref{mutation}.}
 $\mu_{\alpha}(B(K))=B({\bf bm}_\alpha K)$.\\

When we perform a seed mutation we must know how to define the exchange relations. The canonical combinatorial equivalence $\sigma$ helps us clearly understand the correspondence between seeds
 $\Sigma_K$ and $\Sigma_{{\bf bm}_\alpha K}$, so that with the matrix mutation together, we can define the $\binom{h+1}{2}$ exchange relations (see (\ref{e-relation})), giving further the definition of a seed mutation  $\Phi_{\alpha}\colon \Sigma_K\to \Sigma_{{\bf bm}_\alpha K}$ (see also Definition~\ref{seed mutation}). Since all middle-dimensional moves that we perform are bistellar, we call $\Phi_{\alpha}$ the {\em bistellar seed mutation}, and call those $\binom{h+1}{2}$ cluster variables in $\Lambda_\alpha$ and $\Lambda_\beta$ the {\em exchangeable bistellar cluster variables}. We further show that $\Phi_{\alpha}$
is invertible (see Proposition~\ref{b-seed-m}).

Using what we have shown above, we can define the required bistellar cluster algebra $\mathcal{A}_K$ of $K$. Briefly speaking, $\mathcal{A}_K$ is obtained from the initial seed $\Sigma_K$ via bistellar seed mutations by performing all possible
bistellar $h$-moves on $K$. 
To give the definition of $\mathcal{A}_K$  more clearly, we introduce the bistellar exchange graph determined by $K$, which gives the information on
 all possible middle-dimensional bistellar moves on $K$. Moreover, in a similar way to classical cluster algebra, we first define the bistellar cluster algebra $\mathcal{A}_{[K]}$ of $[K]$ (see Definition~\ref{alg1}), where $[K]$ is a set, each of which can be obtained from $K$ by a finite sequence of middle-dimensional bistellar moves. Then we give the definition of
 $\mathcal{A}_K$, which is the reduction of $\mathcal{A}_{[K]}$ to the field $\mathbb{F}\mathcal{X}(K)$ with $\Sigma_K$ as initial seed by using the bistellar seed mutations, where
 $\mathbb{F}\mathcal{X}(K)$ is the field of all rational functions in all cluster variables given by $K$.

 Now our first main result is stated as follows:\\
\\
{\bf Theorem~\ref{main1}.}
{\em Let $K$ and  $K'$ be two connected oriented triangulated closed manifolds of dimension $2h$. If $K'$ is obtained from $K$ by doing a finite sequence of bistellar $h$-moves, then
$\mathcal{A}_{K}\cong \mathcal{A}_{K'}$.}

\begin{remark}
Comparing with classical cluster algebras, our bistellar cluster algebras $\mathcal{A}_K$ have the following properties:
\begin{enumerate}
\item[{\rm (A)}]Each $\mathcal{A}_K$  is of finite type, that is, there are only finitely many cluster variables. In particular, all exchangeable bistellar cluster variables may not cover all the cluster variables.
\item[{\rm (B)}] Generally $\mathcal{A}_K$  is some kind of quotient algebra with an ideal determined by exchange relations. This is an essential difference from classical cluster algebras. The main reason why there is an ideal is because the moves used are bistellar and exchangeable bistellar cluster variables may produce the same exchange relations with each other. So our algebra $\mathcal{A}_K$, even for $\dim K=2$,  does not satisfy the cluster algebra axiom: {\rm every cluster variable in every cluster is exchangeable}.
    Certainly our algebra $\mathcal{A}_K$ is computable (see the example in Section~\ref{example1}).
\end{enumerate}
\end{remark}

Using the bistellar cluster algebras we have constructed, we further consider the problem of PL invariants for closed PL manifolds. Let $K$ be a closed PL manifold of dimension $n$. When $n=2, 4$, we show that a set $[K]$ of all closed PL manifolds PL homeomorphic to $K$ can form  a direct system by associating with the bistellar cluster algebras of closed PL manifolds, and  then show that the limit $\mathcal{A}_K^{PL}$ of this direct system is a PL invariant. This is the second main result of this paper, which is stated as follows:\\
  \\
{\bf Theorem~\ref{main2}.}
{\em Let $K$ and $K'$ be two closed PL manifolds of dimension $n=2,4$, such that they are PL homeomorphic. Then $\mathcal{A}_K^{PL}\cong \mathcal{A}_{K'}^{PL}$.}


\section{Bistellar Moves and Pachner's Theorem}

We begin by reviewing some standard background material regarding the combinatorial topology of triangulated and PL manifolds (see~\cite{bp1} for more information).


\subsection{Triangulated Manifolds and Bistellar Moves}

An \emph{abstract simplicial complex} $K$ on a finite set $S$ is a collection in the power set $2^S$ such that for each $\alpha\in K$, any subset (including the empty set) of $\alpha$ belongs to $K$. Each $\alpha$ in $K$ is called a \emph{simplex} (or a \emph{face}) of $K$ and has dimension $|\alpha|-1$, where
 $|\alpha|$ is the cardinality of $\alpha$. The dimension of $K$ is defined as $\max\{\dim \alpha\mid \alpha\in K\}$.

Let $K_1$ and $K_2$ be two abstract simplicial complexes on vertex sets $S_1$ and $S_2$ respectively.
A map $\phi\colon K_1\to K_2$ is said to be \emph{simplicial} if $\phi$ is an extension of a  map
$f\colon S_1\to S_2$ such that $\phi(\alpha)=f(\alpha)\in K_2$, for any $\alpha\in K_1$. The two simplicial complexes $K_1$ and $K_2$ are said to be \emph{combinatorially equivalent} if there are simplicial maps $\phi\colon  K_1\to K_2$ and $\psi\colon K_2\to K_1$ such that  $\psi\circ\phi=\id_{K_1}$ and $\phi\circ\psi=\id_{K_2}$.

It is well-known (see~\cite{bp1} for example) that there is a bijection, given by geometric realization, between abstract simplicial complexes up to combinatorial equivalence and geometric simplicial complexes up to homeomorphism. Thus, we may identify abstract simplicial complexes with geometric simplicial complexes.

The \emph{join} of two simplicial complexes $K_1$ and $K_2$ on vertex sets $S_1$ and $S_2$ is given by
\[
K_1\ast K_2\,\letbe\, \{ \alpha\in S_1\sqcup S_2\mid \alpha=\alpha_1\cup \alpha_2,\ \alpha_1\in K_1,\ \alpha_2\in K_2\}.
\]
Given a face $\alpha$ of a simplicial complex $K$, the \emph{link} of $\alpha$ in $K$ is the subcomplex given by
\[
\Link_K\alpha\,\letbe\, \{ \alpha'\in K\mid \alpha\cup\alpha'\in K,\ \alpha\cap\alpha'=\emptyset\}.
\]
 A \emph{triangulated manifold} is a simplicial complex $K$ whose geometric realization $|K|$ is a topological manifold.

\begin{definition}[\cite{pa1, bp1}]\label{def-bistellar}
Let $K$ be a triangulated manifold  of dimension $n$. Suppose that $\alpha\in K$ is an $(n-h)$-simplex such that $\Link_K \alpha=\partial \beta$, where $\beta\notin K$. Then the operation
\[
{\bm}_\alpha K\,\letbe\, (K\setminus (\alpha\ast\partial \beta))\cup (\partial \alpha\ast \beta)
\]
is called the {\em bistellar $h$-move at $\alpha$}. Note that $\dim\beta=h$.
\end{definition}

The pair $(\alpha, \beta)$ plays an important role in performing the operation ${\bm}_\alpha K$ and we call the ordered pair $(\alpha, \beta)$ a \emph{bistellar pair} of type $h$ in $K$. Clearly, if $(\alpha, \beta)$ is a bistellar pair of type $h$ in $K$, then  $(\beta, \alpha)$ is a bistellar pair of type $n-h$ in  ${\bm}_\alpha K$ and we can perform a bistellar $(n-h)$-move at $\beta$ on ${\bm}_\alpha K$; in particular,
\[
{\bm}_\beta {\bm}_\alpha K\,=\,K.
\]
Thus, ${\bm}_\beta$ is the inverse operation of ${\bm}_\alpha$. It should be pointed out that there may not be a bistellar pair of a certain type in a triangulated manifold, for example, the boundary complex of an $(n+1)$-simplex has no bistellar pairs of type $h$, for all $h>0$. However, we can always perform a bistellar $0$-move on any triangulated manifold.

It is easy to see that a bistellar $0$-move adds a new vertex, a bistellar $n$-move deletes a vertex, and all other bistellar moves do not change the number of vertices of $K$.

We now study the relation between the numbers of faces of $K$ and ${\bm}_\alpha K$. By $\mathcal{F}_i^K$ we denote the set of all $i$-dimensional simplices of $K$ and also set $\mathcal{F}(K)\letbe \mathcal{F}_{\dim K-1}^K$ for the codimension~1 faces. Each $n$-dimensional simplicial complex $K$ determines an integral vector $f(K)=(f_0, f_1, ..., f_n)$, called the {\em $f$-vector} of $K$, where  $f_i=|\mathcal{F}_i^K|$ is the number of $i$-dimensional simplices in $K$. The $f$-vector of $K$ defines the {\em $h$-vector $h(K)=(h_0, h_1, ..., h_{n+1})$} by
\[
h_0t^{n+1}+h_1t^n+\cdots+ h_{n+1}\,=\,(t-1)^{n+1}+f_0(t-1)^n+\cdots+f_n.
\]
In addition, the $h$-vector of $K$ also defines the \emph{$g$-vector} $g(K)=(g_0,g_1,\dots, g_{\left\lfloor {\frac{n+1}{2}}\right\rfloor })$ by
\[
g_0\,=\,1 \text{ and } g_i\,=\,h_i-h_{i-1},\quad \text{for $i=1,\dots,\left\lfloor {\frac{n+1}{2}}\right\rfloor$.}
\]

\begin{theorem}[Pachner~\cite{pa1}]\label{pachner1}
Let $K$ be a triangulated manifold of dimension $n$. Assume that $(\alpha, \beta)$ is a bistellar pair of type $h$ with $0\leq h\leq \left\lfloor {\frac{n}{2}}\right\rfloor $ in $K$. If $h\leq \left\lfloor {\frac{n-1}{2}}\right\rfloor $, then
\[
g_i({\bm}_\alpha K)\,=\,
\begin{cases}
g_i(K)+1, & \text{if $i=h+1$};\\
g_i(K), & \text{if $i\not= h+1$}.
\end{cases}
\]
If $n=2h$, then $g_i({\bm}_\alpha K)=g_i(K)$ for all $i$.
\end{theorem}

For our purpose we need to further analyze when  $f({\bf bm}_\alpha K)=f(K)$.

\begin{proposition}\label{f-vector}
Let $K$ be a triangulated manifold of dimension $n$. Assume that $(\alpha, \beta)$ is a bistellar pair of type $h$ in $K$. Then $f({\bm}_\alpha K)=f(K)$ if and only if $n=2h$.
\end{proposition}

\begin{proof}
First it is easy to see that if $g_i({\bm}_\alpha K)=g_i(K)$ for all $i$, then $f({\bm}_\alpha K)=f(K)$.
Thus, if $n=2h$, by Theorem~\ref{pachner1} we find that $f({\bm}_\alpha K)=f(K)$.

Assume that $f({\bm}_\alpha K)=f(K)$. By the definition of bistellar moves, the bistellar $h$-move at $\alpha$ on $K$  only changes the local structure of $K$. Actually, since $\Link_K \alpha=\partial \beta$, we have that $\alpha\cap \beta=\emptyset$.  We write $\alpha=(v_0,\dots,v_{n-h})$ and $\beta=(v_{n-h+1},\dots,v_{n+1})$. Then there are the following $h+1$ simplices of dimension $n$ in $K$
\[
F_i\,\letbe\, \alpha\cup\beta\setminus\{v_{n+1-i}\},\quad i=0,\dots, h,
\]
where $\alpha=\bigcap_{i=0}^{h}F_i$, and there are the following $n-h+1$ simplices of dimension $n$ in ${\bm}_\alpha K$
\[
H_j\,\letbe\,\alpha\cup\beta\setminus\{v_j\},\quad j=0,\dots,n-h,
\]
where $\beta=\bigcap_{j=0}^{n-h}H_j$. Since ${\bm}_\alpha K=(K\setminus (\alpha\ast\partial \beta))\cup (\partial \alpha\ast \beta)$, we see that
\[
\mathcal{F}_n^K\setminus\{F_0,\dots,F_{h}\}\,=\,\mathcal{F}_n^{{\bm}_\alpha K}\setminus
\{H_0,\dots,H_{n-h}\}.
\]
Moreover, it follows from $f_n({\bm}_\alpha K)=f_n(K)$ that $h+1=n-h+1,$ so $n=2h$.
\end{proof}

We also see from the proof of Proposition~\ref{f-vector} that $f_n({\bm}_\alpha K)=f_n(K)$ if and only if $n=2h$. Using a similar argument to the proof of Proposition~\ref{f-vector}, we also find that $f_{n-1}({\bm}_\alpha K)=f_{n-1}(K)$ if and only if $n=2h$. Thus we have:

\begin{corollary}\label{f-vector1}
Let $K$ be a triangulated manifold of dimension $n$ and assume that $(\alpha, \beta)$ is a bistellar pair of type $h$ in $K$. Then the following statements are equivalent:
\begin{enumerate}
\item $f({\bm}_\alpha K)=f(K)$;
\item $f_n({\bm}_\alpha K)=f_n(K)$;
\item $f_{n-1}({\bm}_\alpha K)=f_{n-1}(K)$;
\item $n=2h$.
\end{enumerate}
\end{corollary}

In addition, we also  see from the proof of Proposition~\ref{f-vector} that the bistellar $h$-move ${\bm}_\alpha$ on $K$ only changes the subcomplex $\Lambda_\alpha\letbe\alpha\ast \partial\beta$ of $K$,
 determined by the $h+1$ simplices of dimension $n$ in ${\bf F}_\alpha\letbe\{F_0,\dots,F_{h}\}$, which is exactly transformed into the subcomplex ${\bm}_\alpha \Lambda_\alpha\,=\,\Lambda_\beta\,\letbe\, \partial\alpha\ast\beta$ of ${\bm}_\alpha K$, determined by the $n-h+1$ simplices of dimension $n$ in ${\bf H}_\beta\letbe\{H_0,\dots, H_{n-h}\}$.

\begin{corollary}\label{com-equ}
The subcomplexes $\Lambda_\alpha$ and $\Lambda_\beta$ are combinatorially equivalent if and only if $n=2h$.
\end{corollary}

\begin{remark}\label{loc-equ}
If $n=2h$, since $\dim \alpha=\dim\beta=h$, an easy observation shows that for each $0\leq i<h$,
$\mathcal{F}_i^{\Lambda_\alpha}=\mathcal{F}_i^{\Lambda_\beta}$, and  for each $h\leq i\leq n$,
$\mathcal{F}_i^{\Lambda_\alpha}\not=\mathcal{F}_i^{\Lambda_\beta}$ but that
$|\mathcal{F}_i^{\Lambda_\alpha}|=|\mathcal{F}_i^{\Lambda_\beta}|$.
\end{remark}


\subsection{PL Manifolds and Pachner's Theorem}

Two triangulated manifolds $K_1$ and $K_2$ are said to be \emph{PL homeomorphic} if there exist  subdivisions $K'_1$ and $K'_2$ of $K_1$ and $K_2$ such that $K'_1$ and $K'_2$ are  combinatorially equivalent.

 A \emph{PL manifold} of dimension $n$ is a special triangulated manifold  $K$ of dimension $n$ such that $\Link_K \alpha$ is a \emph{PL sphere} of dimension $n-|\alpha|$ for every nonempty simplex $\alpha\in K$, where a PL sphere is a triangulated sphere which is PL homeomorphic to the boundary of a simplex.

The following theorem, due to Pachner, gives a description of PL homeomorphisms in terms of bistellar moves.

\begin{theorem}[Pachner~\cite{pa1, pa2}]\label{home}
Two PL manifolds $K_1, K_2$ are PL homeomorphic if and only if $K_1$ can be transformed into $K_2$ by a finite sequence of bistellar moves.
\end{theorem}

It is well-known that in dimensions $\leq 3$, the categories of PL, topological and smooth manifolds are the same. However, the situation in dimensions $>3$ is quite different. A Theorem of Whitney tells us that a smooth manifold of any dimension always admits a PL structure. In dimension 4, there exist topological manifolds that  admit no PL structure, and even worse, admit no triangulation. However,  the categories of PL and smooth manifolds in dimension 4 agree. Namely, there is exactly one smooth structure on every PL 4-manifold. This means that the classification of smooth 4-manifolds is equivalent to the classification of PL 4-manifolds. But so far, the classification problem is still quite difficult and remains wide open.



\section{Exchange Matrices}

Let $\mathbb{TM}_n$ denote the set of all closed oriented triangulated manifolds of dimension $n>1$. For each closed oriented triangulated manifold $K$, since $K$ is compact, it has finitely many vertices as a simplicial complex. Without the loss of generality, we assume that $K$ is a simplicial complex on vertex set $[m]=\{1,\dots, m\}$, so $K$ is a subcomplex of $2^{[m]}$, where $2^{[m]}$ is  the $(m-1)$-dimensional simplicial complex whose geometric realization is a ball of dimension $m-1$.


\subsection{Pair Ordering}

Using the ordering $1<\dots< m$ on the vertices,  we may impose the  lexographical ordering $<$ on $2^{[m]}$. For example, when $m=3$, we have that
\[
\emptyset\,<\,\{1\}\,<\,\{2\}\,<\,\{3\}\,<\,\{1,2\}\,<\,\{1,3\}\,<\,\{2,3\}\,<\,\{1,2,3\}.
\]
By convention we will always assume a natural  lexographical  ordering $<$ on all faces of $2^{[m]}$.
 For each nonempty subset $\alpha$ in $[m]$, clearly $2^{\alpha}$ admits a  lexographical ordering $<$ which agrees with that on $2^{[m]}$.

For each simplex $\alpha$ with vertices $\{\alpha_0, \alpha_1,\dots, \alpha_n\}\subseteq [m]$, by $(\alpha_0,\alpha_1,\dots, \alpha_n)$ we denote  the simplex $\alpha$ with its orientation determined by the order of its vertices. By $-(\alpha_0,\alpha_1,\dots, \alpha_n)$ we  denote the oriented simplex $(\alpha_0,\alpha_1,\dots, \alpha_n)$ with the opposite orientation. We can also give a \emph{canonical orientation} of $\alpha$ in such a way that the orientation of $\alpha$ agrees with the lexographical ordering of $\{\alpha_0,\alpha_1,\dots, \alpha_n\}$. For example, if $\alpha_0<\alpha_1<\dots<\alpha_n$, then the corresponding oriented simplex  $(\alpha_0,\alpha_1,\dots, \alpha_n)$ has the canonical orientation. Clearly, each simplex in $2^{[m]}$ can be naturally equipped with the canonical orientation.

Let $\alpha=(\alpha_0,\alpha_1,\dots, \alpha_n)$ be an oriented simplex  in $2^{[m]}$  where $n\geq 0$. The boundary operator $\partial$ on $\alpha$ is defined as follows:
\begin{equation}\label{formula1}
\partial (\alpha_0,\alpha_1,\dots, \alpha_n)\,=\,\sum_{i=0}^n(-1)^{i}(\alpha_0,\dots, \widehat{\alpha_i},\dots,\alpha_n).
\end{equation}
We may generalize the formula~\eqref{formula1} as follows: for any $k$ with $1\leq k\leq n+1$,
define
\[
\partial^{(k)}(\alpha_0,\alpha_1,\dots, \alpha_n)\,=\, \sum_{0\leq i_1<\dots <i_k\leq n}(-1)^{i_1+\dots+i_k}\left(\dots, \widehat{\alpha_{i_1}},\dots, \widehat{\alpha_{i_k}},\dots\right).
\]
Obviously $\partial^{(1)}$ is exactly the boundary operator $\partial$ as mentioned in~\eqref{formula1}.

With this understanding, now let us define a \emph{pair ordering} $\prec$ of all $d$-faces in a simplex $\alpha=(\alpha_0,\alpha_1,\dots, \alpha_n)$, where $0\leq d\leq n$.

Note that when $n=0$ or $d=n$, the collection of $d$-faces $\mathcal{F}^{\alpha}_d$ of $\alpha$ only contains a single face, so in this case, we don't need to consider the pair ordering. Without the loss of generality, we assume that  $n\geq 1$ and $0\leq d<n$. In addition, we assume that all faces of dimension less than $n$ are always equipped with the canonical orientation.

 \begin{definition}
Take two $d$-faces  $f$ and $g$ in $\mathcal{F}_d^{\alpha}$ such that $f<g$. If $f\cap g=\emptyset$,  we say that there is no pair ordering between $f$ and $g$. Assume that $f\cap g\ne\emptyset$ and $|f\cap g|=k+1$. Consider the coefficient $c_{f g}$ of the face $(f\cup g)\setminus (f\cap g)$, with its canonical orientation, in the expression of $\partial^{(n-2(d-k)+1)}\alpha$. Then the \emph{pair  ordering} between $f$ and $g$ is defined as follows:
 \begin{align*}
 g\prec f,\quad &\text{if $c_{f g}=+1$;}\\
 f\prec g,\quad &\text{if $c_{f g}=-1$.}
 \end{align*}
 \end{definition}

If $d=0$, then there is no pair ordering on $\mathcal{F}_0^{\alpha}$ since any two $0$-faces in $\mathcal{F}_0^{\alpha}$ have empty intersection. If $\dim \alpha=n=2h$, then any two faces in $\mathcal{F}^{\alpha}_h$ have a nonempty intersection because any middle-dimensional face $f$ in $\alpha$ contains $h+1$ vertices.

\begin{example}\label{ex1}
For example, when $n=2$, take $\alpha=(1, 2, 3)$. In this case, we have that $\mathcal{F}_1^{\alpha}=\{(12), (13), (23)\}$ with the lexographical  ordering $(12)<(13)<(23)$. From the formula $\partial^{(1)} (1,2,3)=(23)-(13)+(12)$, we can obtain the pair orderings for all faces in $\mathcal{F}_1^{\alpha}$ as follows:
\[
(2,3)\prec(1,3);\quad (1,3)\prec(1,2);\quad (1,2)\prec(2,3).
\]
\end{example}

\begin{example}\label{pair ordering}
When $n=3$ we take $\alpha=(1, 2, 3, 4)$. On $\mathcal{F}_2^{\alpha}$, one has the lexographical  ordering
\[
(1,2,3)\,<\,(1,2,4)\,<\,(1,3,4)\,<\,(2,3,4).
\]
From the formula
\begin{align*}
\partial^{(2)} (1,2,3,4)\,=\,-(3,4)+(2,4)-(2,3)-(1,4)+(1,3)-(1,2)
\end{align*}
we can obtain the pair orderings for all 2-faces of $\alpha$ as follows:
\begin{align*}
(1,3,4)\prec(2,3,4);\quad (1, 2,3)\prec(2,3, 4);\quad (2,3,4)\prec(1, 2,4);\\
(1,2,4)\prec (1,3,4);\quad (1,3,4)\prec (1,2,3);\quad (1,2,3)\prec(1,2,4).
\end{align*}
If we take  $\alpha=-(1, 2, 3, 4)$, then all 2-faces of $\alpha$ have the reverse pair orderings.
\end{example}


\subsection{Exchange Matrices}

Throughout the following, we use a convention of orientations and orderings on all simplices in each closed oriented triangulated manifold $K$ (as a subcomplex of $2^{[m]}$) of $\mathbb{TM}_n$. This convention is called the \emph{orientation-ordering convention}, which is stated as follows:
\begin{enumerate}
\item we orient all $n$-dimensional  simplices in $K$ such that their sum is a cycle chain in the chain group $C_n(K)$;
\item all other simplices of dimension less than $n$ are equipped with the canonical orientation as  defined before;
\item there is a natural lexographical  ordering $<$ on all simplices in $K\subseteq 2^{[m]}$.
\end{enumerate}

Given a $K\in \mathbb{TM}_n$, following the idea of~\cite[Definition 4.1]{fst}, we are going to define a skew-symmetric integer square matrix $B(K)$ of size $f_{n-1}(K)=|\mathcal{F}(K)|$ with entries corresponding to all $(n-1)$-dimensional simplices in $K$.

\begin{definition}\label{matrix}
Take any oriented $n$-simplex $\alpha\in \mathcal{F}_n^K$. We first define a skew-symmetric integer square matrix $B^{\alpha}(K)=(b^{\alpha}_{fg})_{f, g\in \mathcal{F}(K)}$ of size $f_{n-1}(K)$ by setting
\[
b^{\alpha}_{fg}\,=\
\begin{cases}
+1, & \text{if $f, g\in \mathcal{F}(\alpha)$ such that $g\prec f$ in $\mathcal{F}(\alpha)$};\\
-1, & \text{if $f, g\in \mathcal{F}(\alpha)$ such that $f\prec g$ in $\mathcal{F}(\alpha)$};\\
\phantom{+}0, & \text{otherwise.}
\end{cases}
\]
 Then the matrix $B(K)$ is defined by
\[
B(K)\,=\,\sum_{\alpha\in\mathcal{F}_n^K}  B^{\alpha}(K).
\]
\end{definition}

Note that the entries of the matrix $B(K)$ belong to $\{-1, 0, +1\}$. Definition~\ref{matrix} is a generalization of~\cite[Definition 4.1]{fst}.

\begin{remark}
For an oriented $n$-simplex $\alpha\in \mathcal{F}_n^K$, if we define $c_{gf}\letbe -c_{fg}$, for $f<g$, then we see that $b_{fg}=c_{fg}$, for all $f,g\in \mathcal{F}(\alpha)$.
\end{remark}

\begin{example}
Let $K$ be the boundary complex of $2^{[5]}$, which is 3-dimensional. Orient the five 3-simplices of $K$ by
\[
\partial^{(1)}(1,2,3,4,5)\,=\,(2,3,4,5)-(1,3,4,5)+(1,2,4,5)-(1,2,3,5)+(1,2,3,4).
\]
So $K$ has the following five oriented 3-simplices
\begin{align*}
\alpha_1\,=\,(2,3,4,5),\ \alpha_2\,=\,-(1,3,4,5),\ \alpha_3\,=\,(1,2,4,5),
\alpha_4\,=\,-(1,2,3,5),\ \alpha_5\,=\,(1,2,3,4)
\end{align*}
such that the sum $\alpha_1+\dots+\alpha_5$ is a cycle chain. Also, $K$ contains ten 2-simplices
\begin{align}\label{ten simplices}
&(1,2,3)\,<\,(1,2,4)\,<\,(1,2,5)\,<\,(1,3,4)\,<\,(1,3,5)\\
<\,&(1,4,5)\,<\,(2,3,4)\,<\,(2,3,5)\,<\,(2,4,5)\,<\,(3,4,5)\nonumber
\end{align}
ordered by the lexographical  ordering, so $B(K)$ is a skew-symmetric integer square matrix of size 10, with all rows and columns ordered by~\eqref{ten simplices}. First we can write out $B^{\alpha_i}(K)$ by using Example~\ref{pair ordering}, and then obtain
\[
B(K)\,=\,\left(
  \begin{array}{cccccccccc}
    0 & -1 & +1 & +1 & -1 & 0 & -1 &+1 & 0 & 0 \\
    +1 & 0 & -1 & -1& 0& +1 & +1 & 0 & -1& 0 \\
    -1 & +1 & 0 & 0& +1 & -1& 0  & -1 &+1 &0\\
    -1 & +1 & 0& 0 & +1 & -1& -1& 0 & 0& +1\\
    +1 & 0 & -1 & -1 & 0 & +1 & 0 & +1 & 0 & -1\\
    0 & -1 & +1 & +1 & -1 & 0 & 0 & 0 & -1 & +1\\
    +1 & -1& 0& +1 & 0 & 0 & 0 & -1 &+1 & -1\\
    -1 & 0 & +1 & 0 & -1 & 0 & +1 & 0 & -1 & +1\\
    0 & +1 & -1 & 0 & 0 & +1 & -1 & +1 & 0 & -1\\
    0 & 0 & 0 & -1 & +1 & -1 & +1 & -1 & +1 & 0
     \end{array}
\right).
\]
\end{example}


\section{Matrix Mutations under Bistellar Moves}

Given a closed oriented triangulated manifold $K$ of $\mathbb{TM}_n$ we regard $K$ as a subcomplex of $2^{[m]}$ where $m$ is the number of vertices of $K$.

Throughout this section, we assume that we can perform a bistellar $h$-move ${\bm}_\alpha$ on $K$ at an $(n-h)$-dimensional simplex $\alpha$, that is, there exists a bistellar pair $(\alpha, \beta)$ such that $\Link_K\alpha=\partial \beta$ but $\beta\not\in K$. We write
\[
L\,\letbe\, {\bf bm}_\alpha K.
\]
By Definition~\ref{matrix}, we see that
\[
B(K)\,=\,(b_{fg}^K)_{f,g\in \mathcal{F}(K)},\quad\text{where we write $\mathcal{F}(K)\letbe \mathcal{F}_{n-1}^K$,}
\]
is of size $f_{n-1}(K)$ and $B(L)=(b_{fg}^{L})_{f,g\in \mathcal{F}(L)}$ is of size $f_{n-1}(L)$. On the other hand, we know from Corollary~\ref{f-vector1} that $f_{n-1}(K)=f_{n-1}(L)$ if and only if $n=2h$. Thus, we will now concentrate on the case $f_{n-1}(K)=f_{n-1}(L)$ and consider how $B(K)$ changes into $B(L)=B({\bm}_\alpha K)$ under the bistellar move ${\bm}_\alpha$.


\subsection{Defining the Matrix Mutation}

In the case of $f_{n-1}(K)=f_{n-1}(L)$, we have $n=2h$, and then $\dim\alpha=\dim\beta=h$.
Without loss of generality, assume that
\[
\alpha\,=\,(v_0,v_1,\dots,v_h) \quad\text{and}\quad \beta\,=\,(v_{h+1},v_{h+2},\dots,v_{n+1}),
\]
and the oriented $n$-simplex $(v_0,v_1,\dots,v_{n})\in K$ is compatible with the given orientations of the $n$-simplices of $K$. Then 
${{\bf F}_\alpha}=\mathcal{F}_{n}^{\Lambda_{\alpha}}$ must consist of
the following $h+1$ oriented simplices of dimension~$n$:
\[
F_i\,=\, (-1)^{i}(v_0,v_1,\dots,v_h,v_{h+1},\dots,\widehat{v_{n+1-i}},\dots,v_{n+1}),\quad \text{for $i=0,\dots, h$,}
\]
such that $\alpha$ is the intersection of the $F_i$, for $i=0,\dots, h$. Since ${\bm}_\alpha$ only changes $\Lambda_\alpha$ into ${\bm}_\alpha \Lambda_\alpha=\Lambda_\beta$ the unchanged $n$-simplices
\[
\mathcal{F}_n^K\setminus {\bf F}_\alpha\,=\,\mathcal{F}_n^{L}\setminus{\bf H}_\beta,
\]
will keep their orientations, and the  $h+1$ oriented simplices of dimension $n$ in ${{\bf H}_\beta}=\mathcal{F}_{n}^{\Lambda_{\beta}}$ will be
\[
H_i\,=\,(-1)^{i}(v_0,v_1,\dots,\widehat{v_i},\dots,v_h,v_{h+1},\dots,v_{n+1}),\quad\text{for $i=0,\dots, h$,}
\]
such that $\beta$ is the intersection of the $H_i$, for $i=0,\dots,h$. This makes sure that
$L={\bm}_\alpha K$ satisfies the orientation-ordering convention.

\begin{lemma}\label{trivial part}
For any $F\in \mathcal{F}_n^K\setminus{\bf F}_\alpha$ and any $f,g\in \mathcal{F}(F)$, we have that $b_{fg}^K=b^{L}_{fg}$. Furthermore,
\[
\sum_{F\in \mathcal{F}_n^K\setminus {\bf F}_\alpha}B^F(K)\,=\,\sum_{F\in \mathcal{F}_n^{L}\setminus{\bf H}_\beta}B^F(L).
\]
\end{lemma}

\begin{proof}
This is because $\mathcal{F}_n^K\setminus {\bf F}_\alpha=\mathcal{F}_n^{L}\setminus{\bf H}_\beta$.
\end{proof}

By Lemma~\ref{trivial part},  it remains to study how $\sum_{F\in {\bf F}_\alpha}B^F(K)$ changes into
$\sum_{F\in {\bf H}_\beta}B^F(L).$ To do this, we need to further analyze the structures of both subcomplexes $\Lambda_\alpha$ and $\Lambda_\beta$. An easy observation shows that both $\Lambda_\alpha$ and $\Lambda_\beta$ have $\binom{h+1}{1}\binom{h+1}{1}+\binom{h+1}{2}$ simplices of dimension $n-1$, and there are $\binom{h+1}{1}\binom{h+1}{1}$ simplices of dimension $n-1$ that lie in both $\Lambda_\alpha$ and $\Lambda_\beta$.  Let
\begin{align*}
\mathcal{D}_{\alpha}\,\letbe\,& \mathcal{F}(\Lambda_{\alpha})\setminus \mathcal{F}(\Lambda_{\beta})\\
\,=\,&\{(v_0,v_1,\dots,v_h,v_{h+1},\dots,\widehat{v_{h+1+i}},\dots,\widehat{v_{h+1+j}},\dots,v_{n+1})\mid 0\leq i<j\leq h\}
\end{align*}
which consists of $\binom{h+1}{2}$ simplices of dimension $n-1$ in $\Lambda_\alpha$. Similarly,
\begin{align*}
\mathcal{D}_{\beta}\,\letbe\,& \mathcal{F}(\Lambda_{\beta})\setminus \mathcal{F}(\Lambda_{\alpha})\\
\,=\,&\{(v_0,v_1,\dots,\widehat{v_i},\dots,\widehat{v_j},\dots,v_h,v_{h+1},\dots,v_{n+1})\mid 0\leq i<j\leq h\}
\end{align*}
consists of $\binom{h+1}{2}$ simplices of dimension $n-1$ in $\Lambda_\beta$. Obviously, $\mathcal{D}_{\alpha}\cap \mathcal{D}_{\beta}=\emptyset$. Then we have that
\begin{equation}\label{compl}
\mathcal{F}(\Lambda_\alpha)\setminus \mathcal{D}_{\alpha}\,=\,\mathcal{F}(\Lambda_\beta)\setminus \mathcal{D}_{\beta}
\end{equation}
which contains exactly $\binom{h+1}{1}\binom{h+1}{1}$  $(n-1)$-simplices.

\begin{lemma}\label{entry}
Suppose $g\notin \mathcal{F}(\Lambda_\alpha)$. Then for any  $f\in \mathcal{D}_{\alpha}$, $b_{fg}^K=0$ in $B(K)$, and for any  $f\in \mathcal{D}_{\beta}$, $b_{fg}^{L}=0$ in $B(L)$.
\end{lemma}

\begin{proof}
This is because each $(n-1)$-simplex in $\mathcal{D}_{\alpha}$ is not a face of any $n$-simplex in $\mathcal{F}_n^K\setminus {{\bf F}_\alpha}$ and  each $(n-1)$-simplex in $\mathcal{D}_{\beta}$ is not a face of any $n$-simplex in $\mathcal{F}_n^{L}\setminus{{\bf H}_\beta}$.
\end{proof}

By Corollary~\ref{com-equ}, $\Lambda_\alpha$ and $\Lambda_\beta$ are combinatorially equivalent.

\begin{lemma}\label{permutation}
Let $\sigma$ be a permutation on $[n+2]$ such that $\sigma(\alpha)=\beta$. Then $\sigma$ induces a combinatorial equivalence between $\Lambda_\alpha$ and $\Lambda_\beta$.
\end{lemma}

\begin{proof}
This is because $\Lambda_\alpha$ and $\Lambda_\beta$ have the same vertex set $[n+2]$, and $\alpha
=\bigcap_{i=0}^{h}F_i$ and $\beta=\bigcap_{i=0}^{h}H_i$.
\end{proof}

\begin{remark}
If $\sigma$ is a permutation on $[n+2]$ such that $\sigma(\alpha)=\beta$, then it is easy to see that $\sigma$ not only gives a bijection from $\mathcal{F}(\Lambda_\alpha)\to \mathcal{F}(\Lambda_\beta)$, but is also invariant on $\mathcal{F}(\Lambda_\alpha)\setminus\mathcal{D}_{\alpha}$ and
maps $\mathcal{D}_{\alpha}$ onto $\mathcal{D}_{\beta}$ bijectively. Note that the permutation $\sigma$ on $[n+2]$ with $\sigma(\alpha)=\beta$ is, of course, not unique.
\end{remark}

We fix the following permutation in the symmetric group $\mathcal{S}_{n+2}$:
\[
\sigma\,=\,
\left(
  \begin{array}{ccccc}
    1 & 2 & \cdots & n+1 & n+2\\
    n+2 & n+1 & \cdots & 2 & 1\\
     \end{array}
\right)
\,\in\, \mathcal{S}_{n+2}.
\]
Note that $\sigma(\alpha)=\beta$.

\begin{definition}\label{mm-def}
We define the \emph{matrix mutation}
\[
\mu_{\alpha}(B(K))\,=\, (\bar{b}_{fg})_{f,g\in \mathcal{F}(L)}
\]
along $\alpha$ as follows:
\[
\bar{b}_{fg}\,=\,
\begin{cases}
-b^K_{\sigma(f)\sigma(g)}, &\text{if $f,g\in \mathcal{F}(\Lambda_\beta)$;}\\
\phantom{-}b^K_{fg}, &\text{otherwise.}
\end{cases}
\]
\end{definition}

\begin{remark}
The permutation $\sigma=
\left(
  \begin{array}{ccccc}
    1 & 2 & \cdots & n+1 & n+2\\
    n+2 & n+1 & \cdots & 2 & 1\\
     \end{array}
\right)$ in $\mathcal{S}_{n+2}$ is of order~2.
\end{remark}

 \begin{proposition}\label{mutation}
 If $L=\bm_{\alpha} K$, where $K\in \mathbb{TM}_{2h}$ and $\alpha\in \mathcal{F}_h^K$, then
 \[
 \mu_{\alpha}(B(K))\,=\, B(L).
 \]
 \end{proposition}

\begin{proof}
Suppose that $f<g$ in $\mathcal{F}(\Lambda_\alpha)$ and that $f\cup g=F_i$, for some $i=0,\dots,h$. Then $\sigma(g)<\sigma(f)$ and $\sigma(F_i)=H_i$. The coefficient $c_{fg}$ appears in $\partial^{(n-1)}F_i$. If $v_i\in f\cap g$ appears in an even position in $F_i$, then so will $\sigma(v_i)$ in $H_i$ (and similarly for odd positions) since $\dim F_i=n$ is even. Therefore, $c_{fg}=c_{\sigma(g)\sigma(f)}$ which gives the result.

If either $f$ or $g$ is not in $\mathcal{F}(\Lambda_\alpha)$, then this follows from Lemma~\ref{trivial part}.
\end{proof}

\begin{remark}
Since ${\bm}_{\beta} L={\bm}_\beta{\bm}_\alpha K=K$, it is easy to see that $\mu_{\alpha}$ has an inverse $\mu_{\beta}$, that is,
\[
\mu_{\beta}(\mu_{\alpha}(B(K)))\,=\, \mu_{\beta}(B(L))\,=\,B(K).
\]
\end{remark}


\subsection{Examples}

Let us now look at the special cases of $h=1$ and $2$.

\subsubsection*{The case $h=1$}

When $h=1$, we have that $\alpha=(1,2)$ and $\beta=(3,4)$, so
\[
{{\bf F}_\alpha}\,=\,\{F_0=(1,2,3),\ F_1=-(1,2,4)\},\quad {{\bf H}_\beta}=\{H_0=(2, 3,4),\ H_1=-(1,3,4)\}
\]
and
\[
\mathcal{F}(\Lambda_\alpha)\,=\,\{(1,2), (1,3), (1,4), (2,3), (2,4)\},\quad
\mathcal{F}(\Lambda_\beta)\,=\,\{(1,3), (1,4), (2,3), (2,4), (3,4)\}.
\]
Also, $\mathcal{D}_{\alpha}=\{(1,2)\}$ and $\mathcal{D}_{\beta}=\{(3,4)\}$. Furthermore, $B(\Lambda_\alpha)$ and $B(\Lambda_\beta)$ determine  nontrivial $5\times 5$ matrices. By a direct calculation, we have that
\[
B(\Lambda_\alpha)\,=\,\left(
  \begin{array}{ccccc}
    0 & 1 & -1 & -1 & 1  \\
    -1 & 0 & 0 & 1 & 0 \\
    1 & 0 & 0 & 0 & -1\\
    1 & -1& 0 & 0 & 0\\
    -1 & 0 & 1 & 0 & 0\\
         \end{array}
\right)
\]
with all rows and columns  ordered by $(1,2), (1,3), (1,4), (2,3),  (2,4)$, and
\[
B(\Lambda_\beta)\,=\,\left(
  \begin{array}{ccccc}
    0 & -1 & 1 & 1 & -1  \\
    1 & 0 & -1 & 0 & 0 \\
    -1 & 1 & 0 & 0 & 0\\
    -1 & 0 & 0 & 0 & 1\\
    1 & 0 & 0 & -1 & 0\\
         \end{array}
\right)
\]
with all rows and columns  ordered by $(3,4), (1,3), (1,4), (2,3), (2,4)$.

Here
 $\sigma=\left(
  \begin{array}{ccccc}
    1 & 2 & 3 & 4  \\
    4 & 3 & 2  & 1 \\
     \end{array}
\right)$, and $\sigma$  maps $\mathcal{F}(\Lambda_\alpha)$ bijectively onto $\mathcal{F}(\Lambda_\beta)$, as shown in the following table:
\begin{center}
\begin{tabular}{|c|c|c|c|c|c|}
\hline &&&&&\\
$f$ & (1,2)& (1,3)& (1,4) &(2,3)& (2,4)\\
\hline
$\sigma(f)$ & (3,4)& (2,4)&(1,4)&(2,3) & (1,3) \\
\hline
\end{tabular}
\end{center}
Then we can read out that
\[
b_{\sigma(f)\sigma(g)}^{L}\,=\, -b_{fg}^K.
\]

\subsubsection*{The case $h=2$}

When $h=2$, we have  that $\alpha=(1, 2, 3)$ and $\beta=(4,5, 6)$, so
\[
{{\bf F}_\alpha}\,=\, \{F_0=(1, 2, 3, 4, 5),\ F_1=-( 1, 2,3,4, 6),\  F_2= (1,2,3,5,6)\}
\]
 and
 \[
 {{\bf H}_\beta}\,=\,\{H_0=(2, 3, 4, 5, 6),\ H_1=-( 1,3,4,5, 6),\  H_2= (1,2,4,5,6)\}.
 \]
Furthermore, $\mathcal{F}(\Lambda_\alpha)$ consists of twelve 3-simplices
 \begin{align*}
&(1,2,3,4), (1,2,3,5), (1,2,3,6), (1,2,4, 5), (1,2,4,6), (1,2,5,6),\\
&(1,3,4,5), (1,3,4,6), (1,3,5,6), (2,3,4,5), (2,3,4,6), (2,3,5,6)
\end{align*}
in $\Lambda_\alpha\subset K$, and $\mathcal{F}(\Lambda_\beta)$ consists of twelve 3-simplices
\begin{align*}
&(1,4,5,6), (2,4,5,6), (3,4,5,6), (1,2,4, 5), (1,2,4,6), (1,2,5,6),\\
&(1,3,4,5), (1,3,4,6), (1,3,5,6), (2,3,4,5), (2,3,4,6), (2,3,5,6)
\end{align*}
in $\Lambda_\beta\subset L$. Also,
\begin{align*}
\mathcal{D}_{\alpha}\,=\,\{(1,2,3,4), (1,2,3,5), (1,2,3,6)\} \quad\text{and}\quad
\mathcal{D}_{\beta}\,=\,\{(1,4, 5, 6), (2, 4, 5,6), (3,4,5,6)\}.
\end{align*}
In this case, $\mathcal{F}(\Lambda_\alpha)\setminus \mathcal{D}_{\alpha}=\mathcal{F}(\Lambda_\beta)\setminus \mathcal{D}_{\beta}$ contains nine $3$-simplices
\begin{align*}
(1,2,4, 5), (1,2,4,6), (1,2,5,6), (1,3,4,5),
(1,3,4,6), (1,3,5,6), (2,3,4,5), (2,3,4,6), (2,3,5,6)
\end{align*}
in $\Lambda_\alpha$ and $\Lambda_\beta$. Both $B(\Lambda_\alpha)$ and $B(\Lambda_\beta)$ determine $12\times 12$ skew-symmetric matrices. A direct calculation shows that
\[
B(\Lambda_\alpha)\,=\,\left(
 \begin{array}{cccccccccccc}
  0 & -1 & 1 & 1 & -1 &  0 &  -1 &  1 &  0 &  1 & -1 &  0 \\
 1 & 0 & -1 & -1 & 0 & 1 & 1 & 0 & -1 & -1 & 0 & 1 \\
 -1 & 1 & 0 & 0 & 1 & -1 & 0 & -1 & 1 & 0 & 1 & -1 \\
 -1 & 1 & 0 & 0 & 0 & 0 & -1 & 0 & 0 & 1 & 0 & 0 \\
 1 & 0 & -1 & 0 & 0 & 0 & 0 & 1 & 0 & 0 & -1 & 0 \\
 0 & -1 & 1 & 0 & 0 & 0 & 0 & 0 & -1 & 0 & 0 & 1 \\
 1 & -1 & 0 & 1 & 0 & 0 & 0 & 0 & 0 & -1 & 0 & 0 \\
 -1 & 0 & 1 & 0 & -1 & 0 & 0 & 0 & 0 & 0 & 1 & 0 \\
 0 & 1 & -1 & 0 & 0 & 1 & 0 & 0 & 0 & 0 & 0 & -1 \\
 -1 & 1 & 0 & -1 & 0 & 0 & 1 & 0 & 0 & 0 & 0 & 0 \\
 1 & 0 & -1 & 0 & 1 & 0 & 0 & -1 & 0 & 0 & 0 & 0 \\
 0 & -1 & 1 & 0 & 0 & -1 & 0 & 0 & 1 & 0 & 0 & 0 \\
 \end{array}\right)
\]
with all rows and columns ordered by $(1,2,3,4)$, $(1,2,3,5)$, $(1,2,3,6)$, $(1,2,4, 5)$, $(1,2,4,6)$, $(1,2,5,6)$, $(1,3,4,5)$, $(1,3,4,6)$, $(1,3,5,6)$, $(2,3,4,5)$, $(2,3,4,6)$, $(2,3,5,6)$, and
\[
B(\Lambda_\beta)\,=\,\left(
 \begin{array}{cccccccccccc}
  0 & -1 & 1 & 1 & -1 &  1 &  -1 &  1 &  -1 &  0 & 0 &  0 \\
 1 & 0 & -1 & -1 & 1 & -1 & 0 & 0 & 0 & 1 & -1 & 1 \\
 -1 & 1 & 0 & 0 & 0 & 0 & 1 & -1 & 1 & -1 & 1 & -1 \\
 -1 & 1 & 0 & 0 & -1 & 1 & 0 & 0 & 0 & 0 & 0 & 0 \\
 1 & -1 & 0 & 1 & 0 & -1 & 0 & 0 & 0 & 0 & 0 & 0 \\
 -1 & 1 & 0 & -1 & 1 & 0 & 0 & 0 & 0 & 0 & 0 & 0 \\
 1 & 0 & -1 & 0 & 0 & 0 & 0 & 1 & -1 & 0 & 0 & 0 \\
 -1 & 0 & 1 & 0 & 0 & 0 & -1 & 0 & 1 & 0 & 0 & 0 \\
 1 & 0 & -1 & 0 & 0 & 0 & 1 & -1 & 0 & 0 & 0 & 0 \\
 0 & -1 & 1 & 0 & 0 & 0 & 0 & 0 & 0 & 0 & -1 & 1 \\
 0 & 1 & -1 & 0 & 0 & 0 & 0 & 0 & 0 & 1 & 0 & -1 \\
 0 & -1 & 1 & 0 & 0 & 0 & 0 & 0 & 0 & -1 & 1 & 0 \\
 \end{array}\right)
\]
with all rows and columns ordered by $(1,4, 5, 6)$, $(2, 4, 5,6)$, $(3,4,5,6)$, $(1,2,4, 5)$, $(1,2,4,6)$, $(1,2,5,6)$, $(1,3,4,5)$, $(1,3,4,6)$, $(1,3,5,6)$, $(2,3,4,5)$, $(2,3,4,6)$, $(2,3,5,6)$.

The permutation $\sigma=\left(
  \begin{array}{cccccc}
    1 & 2 & 3 & 4 & 5 & 6 \\
    6 & 5 & 4 & 3 & 2 & 1 \\
     \end{array}
\right)$ then maps $\mathcal{F}(\Lambda_\alpha)$ bijectively to $\mathcal{F}(\Lambda_\beta)$ similarly to the previous case. Then we can easily read out that
\[
b_{\sigma(f)\sigma(g)}^{L}\,=\, -b_{fg}^K,\quad \text{for $f, g\in \mathcal{F}(\Lambda_\alpha)$.}
\]


\section{Bistellar Cluster Algebras}

This section is dedicated to defining a class of algebras from closed oriented triangulated manifolds, which we call \emph{bistellar cluster algebras}.

Given a closed oriented triangulated manifold $K\in \mathbb{TM}_{n}$,  assume that $K$ is a simplicial complex on the vertex set $[m]$, so $K$ is a subcomplex of $2^{[m]}$. We shall mainly be concerned with the case in which $n$ is even and study the algebraic behaviour of middle-dimensional bistellar moves on $K$. We write $n=2h$ and assume that we can perform a bistellar $h$-move on $\alpha\in K$ and set $L\letbe {\bm}_{\alpha} K$.

Now we are going to construct an algebra $\mathcal{A}_{K}$ associated to $K$. Our construction differs from the one of a classical cluster algebra (see~\cite{fz1, fst}), but the basic idea of our construction follows closely.


\subsection{Seeds and Bistellar Seed Mutations}

Let $(\mathbb{P}, \oplus, \cdot)$ be a semifield, that is, $(\mathbb{P}, \oplus)$ is a commutative semigroup and $(\mathbb{P}, \cdot)$ is an abelian group such that multiplication distributes over the auxiliary addition $\oplus$. By $\mathbb{Z}\mathbb{P}$ we denote the group ring of $(\mathbb{P}, \cdot)$ with integer coefficients.

\begin{definition}
Let $K$ be a closed oriented triangulated manifold in $\mathbb{TM}_{n}$. Let $\mathbb{F}\mathcal{X}(K)$ denote the field of rational functions over $\mathbb{ZP}$ with independent variables given by $\mathcal{X}(K)$, where  $\mathcal{X}(K)=\{x_f\mid f\in \mathcal{F}(K)\}$ is called a \emph{cluster}.

Associated to $K$, we get a triple $\Sigma_K=(\mathcal{X}(K), {\bf p}_K, B(K))$, called a  \emph{seed}, where
 \begin{enumerate}
\item $\mathcal{X}(K)$ is the cluster corresponding to $K$,  as defined above;
\item the coefficient tuple ${\bf p}_K=(p^\pm_{K,x})_{x\in \mathcal{X}(K)}$  is a $2d$-tuple of elements of $\mathbb{P}$
satisfying the normalization condition $p_{K,x}^+\oplus p_{K,x}^-=1$ for all $x\in \mathcal{X}(K)$, where $d=f_{n-1}(K)$;
\item $B(K)$ is the exchange matrix of $K$.
\end{enumerate}
 \end{definition}

First let us analyse the behaviour of the  seeds  under the actions of bistellar $h$-moves. Let $(\alpha, \beta)$ be a bistellar pair in $K$ of type $h$. Set $\mathcal{X}(\mathcal{D}_{\alpha})=\{x_f\mid f\in \mathcal{D}_{\alpha}\}$  and $\mathcal{X}(\mathcal{D}_{\beta})=\{x_f\mid f\in \mathcal{D}_{\beta}\}$. Then $\mathcal{X}(\mathcal{D}_{\alpha})$ and $\mathcal{X}(\mathcal{D}_{\beta})$ are subsets of $\mathcal{X}(K)$ and $\mathcal{X}(L)$, respectively. Note that by Corollary~\ref{f-vector1},  we have $f_{n-1}(L)=f_{n-1}(K)=d$, so $|\mathcal{X}(K)|=|\mathcal{X}(L)|$.

\begin{definition}
Elements in $\mathcal{X}(\mathcal{D}_{\alpha})$ and $\mathcal{X}(\mathcal{D}_{\beta})$ are said to be \emph{exchangeable bistellar cluster variables} in $\mathcal{X}(K)$ and $\mathcal{X}(L)$, respectively.
\end{definition}

With this convention,  for the bistellar pair $(\alpha, \beta)$ in $K$, we have the following relations between the  seeds $\Sigma_K=(\mathcal{X}(K), {\bf p}_K, B(K))$ and $\Sigma_{L}=(\mathcal{X}(L), {\bf p}_{L}, B(L))$ as follows:
\begin{enumerate}
\item Firstly, we  see that between the two clusters $\mathcal{X}(K)$ and $\mathcal{X}(L)$ we have the following:
\[
\mathcal{X}(L)\,=\,(\mathcal{X}(K)\setminus \mathcal{X}(\mathcal{D}_\alpha))\cup \mathcal{X}(\mathcal{D}_\beta)
\quad\text{and}\quad
\mathcal{X}(K)\,=\,(\mathcal{X}(L)\setminus \mathcal{X}(\mathcal{D}_{\beta}))\cup \mathcal{X}(\mathcal{D}_{\alpha}),
\]
which exactly corresponds to
\[
L\,=\,{\bm}_\alpha K\,=\,(K\setminus \Lambda_{\alpha})\cup \Lambda_{\beta}
\quad\text{and}\quad
K\,=\, \bm_{\beta} L\,=\,(L\setminus \Lambda_{\beta})\cup \Lambda_{\alpha}
\]
respectively. Since  $\sigma(\mathcal{F}(\Lambda_\alpha))=\mathcal{F}(\Lambda_\beta)$, we can use  $\sigma$ to define a bijection $\varphi^{K}_{\alpha}\colon \mathcal{X}(K)\to \mathcal{X}(L)$ by
    \begin{equation}\label{cc}
    \varphi^{K}_{\alpha}(x_f)\,=\,
    \begin{cases}
    x_{\sigma(f)}, &\text{if $f\in \mathcal{F}(\Lambda_\alpha)$;}\\
    x_{f}, & \text{if $f\notin  \mathcal{F}(\Lambda_\alpha)$,}
    \end{cases}
    \end{equation}
    satisfying $\varphi^{K}_{\alpha}(\mathcal{X}(\mathcal{D}_{\alpha}))=\mathcal{X}(\mathcal{D}_{\beta})$ since $\sigma(\mathcal{D}_{\alpha})=\mathcal{D}_{\beta}$. Conversely, since $\mathcal{F}(K)\setminus \mathcal{F}(\Lambda_\alpha)=\mathcal{F}(L)\setminus \mathcal{F}(\Lambda_\beta)$ and $\sigma$ is an involution, we can also use $\sigma$ to define  a bijection $\varphi^{L}_{\beta}\colon \mathcal{X}(L)\rightarrow\mathcal{X}(K)$ similarly
     satisfying $\varphi^{L}_{\beta}(\mathcal{X}(\mathcal{D}_{\beta}))=\mathcal{X}(\mathcal{D}_{\alpha})$. Clearly, $\varphi^{L}_{\beta}$ is the inverse of $\varphi^{K}_{\alpha}$. So $\varphi^{L}_{\beta}(\varphi^{K}_{\alpha}(\mathcal{X}(K)))\,=\,\mathcal{X}(K)$.

   \item Secondly, $\varphi^{K}_{\alpha}$ and $\varphi^{L}_{\beta}$ determine two bijections
      \[
      c_K\colon {\bf p}_K\longrightarrow {\bf p}_{L}\quad
     \text{given by $p^\pm_{K, x}\longmapsto p^\pm_{L, \varphi^{K}_{\alpha}(x)}$}
     \]
     and
\[
c_{L}\colon {\bf p}_{L}\longrightarrow {\bf p}_{K}\quad
     \text{given by $p^\pm_{L, x}\longmapsto p^\pm_{K, \varphi^{L}_{\beta}(x)}$}.
     \]
Obviously, $c_{L}(c_K({\bf p}_K))={\bf p}_K$.

\item Finally,  by Proposition~\ref{mutation}, the  matrix mutation $\mu_{\alpha}$ satisfies the property that $\mu_{\alpha}(B(K))=B(L)$. We have, for each $b_{fg}^K\in B(K)$,
\begin{equation}\label{mm}
\mu_{\alpha}(b_{fg}^K)\,=\,
\begin{cases}
 -b_{fg}^K=b_{\sigma(f)\sigma(g)}^L, & \text{if $f, g\in \mathcal{F}(\Lambda_{\alpha})$};\\
 \phantom{+}b_{fg}^K=b_{fg}^L, & \text{otherwise}.
\end{cases}
\end{equation}
Since ${\bm}_{\beta}L={\bm}_\beta{\bm}_\alpha K=K$ and $\sigma$ is an involution, we have that $\mu_{\beta}(B(L)=B(K)$, so
\[
\mu_{\beta}(\mu_{\alpha}(B(K)))\,=\,B(K).
\]
\end{enumerate}

We now set up some notation. Suppose that $K\in \mathbb{TM}_n$ with bistellar pair $(\alpha, \beta)$ of type $h$ with $n=2h$. Then for $f\in \mathcal{F}(K)$ we write
\[
u_f^K\,\letbe\, p_{K,x_f}^+/p_{K,x_f}^-.
\]
It then follows from $p^+_{K, x_f}\oplus p^-_{K, x_f}=1$ that
\[
p^+_{K, x_f}\,=\,\frac{u_f^{K}}{1\oplus u_f^{K}}\quad \text{and}\quad p^-_{K, x_f}\,=\,\frac{1}{1\oplus u_f^{K}}.
\]
We also write
\[
\pi_{\alpha, f}^K\,\letbe\,\prod\limits_{g\in \mathcal{D}_{\alpha}}(p^+_{K, x_{g}})^{[b^K_{gf}]_+} (p^-_{K, x_{g}})^{-[-b^K_{gf}]_+},
\]
where $[b]_+\letbe \max\{b,0\}$.

    \begin{definition}\label{seed mutation}
    With the above understanding, we define the
    \emph{bistellar seed mutation}
    \[
    \Phi_{\alpha}\colon \Sigma_K\,=\, (\mathcal{X}(K), {\bf p}_K, B(K))\longrightarrow \Sigma_{L}\,=\,(\mathcal{X}(L), {\bf p}_{L}, B(L))
    \]
    as follows:
     \begin{enumerate}
\item\label{Mon} On $\mathcal{X}(K)$, $\varphi^{K}_{\alpha}(\mathcal{X}(K))=\mathcal{X}(L)=(\mathcal{X}(K)\setminus \mathcal{X}(\mathcal{D}_{\alpha}))\cup \mathcal{X}(\mathcal{D}_{\beta})$ such that there are \emph{exchange relations} between exchangeable bistellar cluster variables of $\mathcal{X}(\mathcal{D}_{\alpha})$ and  $\mathcal{X}(\mathcal{D}_{\beta})$ as follows: for $f\in \mathcal{D}_{\alpha}$,
\begin{multline}\label{e-relation}
x_f x_{\sigma(f)}\,=\, \frac{1}{D^K_{\alpha, f}}
\left(p^+_{K, x_f} \lcm\left( \prod_{g\in \mathcal{F}(\Lambda_{\alpha})\setminus \mathcal{D}_{\alpha}} x_g^{[b_{fg}^K]_+}, \prod_{g\in \mathcal{F}(\Lambda_{\alpha})\setminus \mathcal{D}_{\alpha}} x_{\sigma(g)}^{[b_{fg}^K]_+}\right)\right.\\
+\left.p^-_{K,x_f} \lcm\left( \prod\limits_{g\in \mathcal{F}(\Lambda_{\alpha})\setminus \mathcal{D}_{\alpha}} x_g^{[-b_{fg}^K]_+}, \prod\limits_{g\in \mathcal{F}(\Lambda_{\alpha})\setminus \mathcal{D}_{\alpha}} x_{\sigma(g)}^{[-b_{fg}^K]_+}\right)\right)
\end{multline}
and
 \begin{equation}\label{e:Dpm}
 D^K_{\alpha, f}\,\letbe\, \prod_{\substack{g\in \mathcal{F}(\Lambda_{\alpha})\setminus \mathcal{D}_{\alpha}\\ b_{fg}^K>0,\ b_{f\sigma(g)}^K<0}} x_g^{b_{fg}^K} x_{\sigma(g)}^{b_{fg}^K}\,=\, \prod_{\substack{g\in \mathcal{F}(\Lambda_{\alpha})\setminus \mathcal{D}_{\alpha}\\ b_{fg}^K<0,\ b_{f\sigma(g)}^K>0}} x_g^{-b_{fg}^K} x_{\sigma(g)}^{-b_{fg}^K},
 \end{equation}
 which follows from $\sigma(\mathcal{F}(\Lambda_{\alpha})\setminus \mathcal{D}_{\alpha})=\mathcal{F}(\Lambda_{\alpha})\setminus \mathcal{D}_{\alpha}$ and $\sigma$ being of order~2.

There are $\binom{h+1}{2}$ exchange relations since $|\mathcal{X}(\mathcal{D}_{\alpha})|=|\mathcal{X}(\mathcal{D}_{\beta})|=\binom{h+1}{2}$.

\item\label{seedp} On ${\bf p}_K$,  $c_K({\bf p}_K)={\bf p}_{L}=(p^\pm_{L, x_f})_{f\in \mathcal{F}(L)}$ is uniquely determined by the normalization condition
    $p^+_{L, x_f}\oplus p^-_{L, x_f}=1$
    together with
    \[
    u^L_f\,=\,
    \begin{cases}
    \left(u^K_{\sigma(f)}\right)^{-1},  & \text{ if $f\in \mathcal{D}_{\beta}$;}\\
 \pi_{\alpha, \sigma(f)}^K u^K_{\sigma(f)}, & \text{if $f\in \mathcal{F}(\Lambda_{\beta})\setminus\mathcal{D}_{\beta}$};\\
     u^K_f, & \text{ otherwise.}
    \end{cases}
    \]
\item On $B(K)$, $\mu_{\alpha}(B(K))=B(L)$.
\end{enumerate}
\end{definition}

\begin{remark}
Since ${\bm}_\beta L={\bm}_\beta{\bm}_\alpha K=K$, in a similar way as above we also can define the bistellar seed mutation
\[
\Phi_{\beta}\colon  \Sigma_{L}\,=\,(\mathcal{X}(L), {\bf p}_{L}, B(L))\longrightarrow\Sigma_K\,=\, (\mathcal{X}(K), {\bf p}_K, B(K)).
\]
\end{remark}

From Definition~\ref{seed mutation}\eqref{Mon}, we write
\[
M^+_{K, \alpha, f}\,=\, \frac{p^+_{K, x_f}}{D^K_{\alpha, f}} \lcm\left( \prod_{g\in \mathcal{F}(\Lambda_{\alpha})\setminus \mathcal{D}_{\alpha}} x_g^{[b_{fg}^K]_+}, \prod_{g\in \mathcal{F}(\Lambda_{\alpha})\setminus \mathcal{D}_{\alpha}} x_{\sigma(g)}^{[b_{fg}^K]_+}\right)
\]
and
\[
M^-_{K, \alpha, f}\,=\, \frac{p^-_{K, x_f}}{D^K_{\alpha, f}} \lcm\left( \prod_{g\in \mathcal{F}(\Lambda_{\alpha})\setminus \mathcal{D}_{\alpha}} x_g^{[-b_{fg}^K]_+}, \prod_{g\in \mathcal{F}(\Lambda_{\alpha})\setminus \mathcal{D}_{\alpha}} x_{\sigma(g)}^{[-b_{fg}^K]_+}\right),
\]
so $x_fx_{\sigma(f)}=M^++M^-$.

Then it is easy to see, using equation~\eqref{e:Dpm}, that:

\begin{lemma}\label{pro1}
The two monomials $M^+_{K,\alpha,f}$ and $M^-_{K,\alpha,f}$ have no common divisor.
\end{lemma}

\begin{lemma}\label{pro2}
Following from Definition~\ref{seed mutation}\eqref{seedp} we have the following relations:
\[
p^+_{L,x_f}\,=\,
\begin{cases}
p^-_{K,\sigma(f)}, &\text{if $f\in\mathcal{D}_\beta$};\\
\frac{\pi_{\alpha, \sigma(f)}^K u^K_{\sigma(f)}}{1\oplus \pi_{\alpha, \sigma(f)}^K u^K_{\sigma(f)}}, &\text{if $f\in\mathcal{F}(\Lambda_\beta)\setminus \mathcal{D}_\beta$};\\
p^+_{K, x_f}, &\text{otherwise},
\end{cases}
\quad\text{and}\quad
p^-_{L,x_f}\,=\,
\begin{cases}
p^+_{K,\sigma(f)}, &\text{if $f\in\mathcal{D}_\beta$};\\
\frac{1}{1\oplus \pi_{\alpha, \sigma(f)}^K u^K_{\sigma(f)}}, &\text{if $f\in\mathcal{F}(\Lambda_\beta)\setminus \mathcal{D}_\beta$};\\
p^-_{K, x_f}, &\text{otherwise}.
\end{cases}
\]
\end{lemma}

\begin{proof}
If $f\in \mathcal{D}_{\beta}$, since $p^+_{L, x_f}/ p^-_{L, x_f}=p^-_{K, x_{\sigma(f)}}/p^+_{K, x_{\sigma(f)}}$ and $p^+_{K, x_{\sigma(f)}}\oplus p^-_{K, x_{\sigma(f)}}=1$, then $p^+_{K, x_{\sigma(f)}}= \frac{1}{1\oplus u_f^{L}}$ and $p^-_{K, x_{\sigma(f)}}=\frac{u_f^{L}}{1\oplus u_f^{L}}$, so $p^\pm_{L, x_f}=p^\mp_{K, x_{\sigma(f)}}$ as required.

If $f\in \mathcal{F}(L)\setminus \mathcal{F}(\Lambda_\beta)$, since $p^+_{L, x_f}/ p^-_{L, x_f}=p^+_{K, x_{\sigma(f)}}/p^-_{K, x_{\sigma(f)}}$, then clearly $p^\pm_{L, x_f}=p^\pm_{K, x_{f}}$.

If $f\in \mathcal{F}(\Lambda_{\beta})\setminus\mathcal{D}_{\beta}$, by~\eqref{mm}, by direct calculation
\[
p^+_{L, x_f}\,=\, \frac{u^L_f}{1\oplus u^L_f}\,=\, \frac{\pi^K_{\alpha, \sigma(f)}u^K_{\sigma(f)}}{1\oplus \pi^K_{\alpha, \sigma(f)}u^K_{\sigma(f)}}
\quad\text{and}\quad
p^-_{L, x_f}\,=\, \frac{1}{1\oplus u^L_f}\,=\, \frac{1}{1\oplus \pi^K_{\alpha, \sigma(f)}u^K_{\sigma(f)}}.
\]
\end{proof}

\begin{remark}
We also have that
\begin{align*}
\pi_{\alpha, \sigma(f)}^K\,&=\,\prod\limits_{g\in \mathcal{D}_{\alpha}}(p^+_{K, x_{g}})^{[b^K_{g\sigma(f)}]_+} (p^-_{K, x_{g}})^{-[-b^K_{g\sigma(f)}]_+}\,=\, \prod\limits_{g\in \mathcal{D}_{\beta}}(p^+_{K, x_{\sigma(g)}})^{[b^K_{\sigma(g)\sigma(f)}]_+} (p^-_{K, x_{\sigma(g)}})^{-[-b^K_{\sigma(g)\sigma(f)}]_+}\\
&=\, \prod\limits_{g\in \mathcal{D}_{\beta}}(p^-_{L, x_g})^{[-b^{L}_{gf}]_+} (p^+_{L, x_{g}})^{-[b^{L}_{gf}]_+}\,=\, \left(\prod\limits_{g\in \mathcal{D}_{\beta}}(p^+_{L, x_{g}})^{[b^{L}_{gf}]_+}(p^-_{L, x_g})^{-[-b^{L}_{gf}]_+}\right)^{-1}\,=\, \left(\pi^L_{\beta, f}\right)^{-1}.
\end{align*}
\end{remark}

\begin{lemma} \label{pro3}
For $f\in \mathcal{D}_{\alpha}$, $M^\pm_{K,\alpha,f}=M^\mp_{L,\beta,\sigma(f)}$. Furthermore,
\[
M^+_{L,\beta,\sigma(f)}+ M^-_{L,\beta,\sigma(f)}\,=\,M^-_{K,\alpha,f}+M^+_{K,\alpha,f}\,=\,x_fx_{\sigma(f)}.
\]
\end{lemma}

\begin{proof}
For $f, g\in \mathcal{F}(\Lambda_{\alpha})$,  we have by~\eqref{mm} that $b_{\sigma(f)\sigma(g)}^{L}=-b^K_{fg}$. For $f\in \mathcal{D}_{\alpha}$,
by Lemma~\ref{pro2} we have that $p^\pm_{L,  x_{\sigma(f)} }=p^\mp_{K, x_f}$.
We note that $\sigma$ is an involution,  $\sigma(\mathcal{F}(\Lambda_{\alpha})\setminus\mathcal{D}_{\alpha})=\mathcal{F}(\Lambda_{\alpha})\setminus\mathcal{D}_{\alpha}=\mathcal{F}(\Lambda_{\beta})\setminus\mathcal{D}_{\beta}$ and $\sigma(\mathcal{D}_{\alpha})=\mathcal{D}_{\beta}$.
Thus, for  $f\in \mathcal{D}_{\alpha}$,
\begin{align*}
 M^\mp_{L,\beta,\sigma(f)} \,=\,&\  \frac{p^\mp_{L,x_{\sigma(f)}}}{D^{L}_{\beta,\sigma(f)}} \lcm \left(\prod\limits_{g\in \mathcal{F}(\Lambda_{\beta})\setminus\mathcal{D}_{\beta}}x_g^{[\mp b^{L}_{\sigma(f)g}]_+},
\prod\limits_{g\in \mathcal{F}(\Lambda_{\beta})\setminus\mathcal{D}_{\beta}}x_{\sigma(g)}^{[\mp b^{L}_{\sigma(f)g}]_+}\right)\\
\,=\,& \frac{p^\pm_{K, x_f}}{D^{L}_{\beta,\sigma(f)}} \lcm \left(\prod\limits_{g\in \mathcal{F}(\Lambda_{\alpha})\setminus\mathcal{D}_{\alpha}}x_{\sigma(g)}^{[\mp b^{L}_{\sigma(f)\sigma(g)}]_+},
\prod\limits_{g\in \mathcal{F}(\Lambda_{\alpha})\setminus\mathcal{D}_{\alpha}}x_{\sigma^2(g)}^{[\mp b^{L}_{\sigma(f)\sigma(g)}]_+}\right)\\
\,=\,& \frac{p^\pm_{K, x_f}}{D^{L}_{\beta,\sigma(f)}}\lcm \left(\prod\limits_{g\in \mathcal{F}(\Lambda_{\alpha})\setminus\mathcal{D}_{\alpha}}x_{g}^{[\pm b^K_{fg}]_+},
\prod\limits_{g\in \mathcal{F}(\Lambda_{\alpha})\setminus\mathcal{D}_{\alpha}}x_{\sigma(g)}^{[\pm b^K_{fg}]_+}\right)
\end{align*}
where
\[
 D^L_{\beta, \sigma(f)}\,=\, \prod_{\substack{g\in \mathcal{F}(\Lambda_{\beta})\setminus \mathcal{D}_{\beta}\\ b_{\sigma(f)g}^L>0,\ b_{\sigma(f)\sigma(g)}^L<0}} x_g^{b_{\sigma(f)g}^L} x_{\sigma(g)}^{b_{\sigma(f)g}^L}
 \,=\, \prod_{\substack{g\in \mathcal{F}(\Lambda_{\alpha})\setminus \mathcal{D}_{\alpha}\\ -b_{fg}^K>0,\ -b_{f\sigma(g)}^K<0}} x_{\sigma(g)}^{-b_{fg}^K} x_{g}^{-b_{fg}^K}
 \,=\, D^K_{\alpha,f}
\]
using equation~\eqref{e:Dpm}. Thus, we have that  $M^\pm_{K,\alpha,f}=M^\mp_{L,\beta,\sigma(f)}$.
\end{proof}

The above arguments imply that:
\begin{proposition}\label{b-seed-m}
The bistellar seed mutation $\Phi_{\alpha}\colon \Sigma_K\longrightarrow \Sigma_{L}$ is invertible, and its inverse is $\Phi_{\beta}$.
\end{proposition}

The seed mutation $\Phi_{\alpha}$ also determines a natural
$\mathbb{ZP}$-linear field isomorphism $\mathcal{L}_{(K,L)}\colon \mathbb{F}\mathcal{X}(K)\longrightarrow\mathbb{F}\mathcal{X}(L)$ given by
\begin{equation}\label{field-iso}
\mathcal{L}_{(K,L)}(x_f)\,=\,\varphi^K_{\alpha}(x_f)\,=\, \begin{cases}
    x_f & \text{if $f\in \mathcal{F}(K)\setminus \mathcal{F}(\Lambda_\alpha)$;}\\
    x_{\sigma(f)} & \text{if $f\in  \mathcal{F}(\Lambda_\alpha)\setminus\mathcal{D}_{\alpha}$;}\\
    x_{\sigma(f)}=\frac{M^+_{K,\alpha,f}+M^-_{K,\alpha,f}}{x_f} & \text{if $f\in \mathcal{D}_{\alpha}$.}
    \end{cases}
    \end{equation}
Clearly, $\mathcal{L}_{(K, L)}$ is invertible and $\mathcal{L}_{(L, K)}=\mathcal{L}_{(K, L)}^{-1}$.
Since $\mathcal{X}(K)\setminus \mathcal{X}(\mathcal{D}_{\alpha})=\mathcal{X}(L)\setminus \mathcal{X}(\mathcal{D}_{\beta})$ and
\[
\frac{M^+_{K,\alpha,f}+M^-_{K,\alpha,f}}{x_f}\,\in\, \mathbb{F}\mathcal{X}(L)
\]
for $f\in \mathcal{D}_{\alpha}$, this gives:
\begin{corollary}\label{field-1}
Under the action of the seed mutation $\Phi_{\alpha}$ (or $\Phi_{\beta}$),   $\mathbb{F}\mathcal{X}(K)$ can be identified with
$\mathbb{F}\mathcal{X}(L)$.
\end{corollary}

We can now actually define the required bistellar cluster algebra $\mathcal{A}_K$ from the initial seed $\Sigma_K$ via bistellar seed mutations by performing all possible bistellar $h$-moves on $K$, as in the definition of a classical cluster algebra. To understand the structure of  $\mathcal{A}_K$ more clearly, we would like to give more details in the next subsection.


\subsection{Bistellar Exchange Graphs and Bistellar Cluster Algebras}

We define an equivalence relation on $\mathbb{TM}_{2h}$ by saying that $K_1\sim K_2$ if and only if $K_2$ can be obtained from $K_1$ by a finite sequence of bistellar $h$-moves (including the empty sequence). This is clearly an equivalence relation and we denote by $[K]$ the equivalence class containing $K$.

Note that it is possible that $[K]$ only contains $K$ itself. For example, if $K$ is the boundary complex of an $(2h+1)$-simplex, then we cannot perform any bistellar $h$-moves on $K$, so $[K]=\{K\}$. If $|[K]|>1$, then by Corollary~\ref{f-vector1}, any triangulated manifold $L\in [K]$ is also a subcomplex of of $2^{[m]}$, so $[K]$ is finite. Throughout the following, we assume that $|[K]|>1$.

\begin{definition}
A finite sequence of pairs $(\alpha_1, \beta_1),\dots, (\alpha_l, \beta_l)$ in $\mathcal{F}_h^{2^{[m]}}$ is called a \emph{bistellar sequence} of $[K]$ if there exists $L', L''\in [K]$ such that $L''={{\bm}_{\alpha_{l}}\cdots{\bm}_{\alpha_1} L'}$.

A pair $(\alpha, \beta)$ in $\mathcal{F}_h^{2^{[m]}}$ is said to be a \emph{bistellar pair} of $[K]$ if  there is a bistellar sequence of $[K]$, $(\alpha_1, \beta_1),\dots, (\alpha_l, \beta_l)$, such that $(\alpha, \beta)$ is one of $(\alpha_1, \beta_1),\dots, (\alpha_l, \beta_l)$.
\end{definition}

Obviously, if $(\alpha_1, \beta_1),\dots, (\alpha_l, \beta_l)$ is a bistellar sequence of $[K]$, then so is its reverse.

By $\mathcal{S}_{bp}^{[K]}$ we denote the set  of all bistellar pairs of $[K]$, and for each $L\in [K]$, by $\mathcal{S}_{bp}^L$ we denote the subset of $\mathcal{S}_{bp}^{[K]}$ consisting of bistellar pairs in $L$. Obviously, each $\mathcal{S}_{bp}^L$ is nonempty. On the other hand, given a pair $(\alpha, \beta)\in \mathcal{S}_{bp}^{[K]}$, it is easy to see that there must be a triangulated manifold $L\in [K]$ such that $(\alpha, \beta)$ is a bistellar pair in $L$, that is, $(\alpha, \beta)\in \mathcal{S}_{bp}^L$. Thus,
\[
\mathcal{S}_{bp}^{[K]}\,=\,\bigcup_{L\in [K]}\mathcal{S}_{bp}^L.
\]
Note that generally, for two $L, L'\in [K]$, the intersection $\mathcal{S}_{bp}^L\cap \mathcal{S}_{bp}^{L'}$ may not be empty.

\begin{lemma}\label{bpset}
The set of bistellar pairs $\mathcal{S}_{bp}^{[K]}$ has the following properties:
\begin{enumerate}
\item $\mathcal{S}_{bp}^{[K]}$ is a finite set.
\item $\mathcal{S}_{bp}^{[K]}$ is nonempty if and only if $|[K]|>1$.
\item  $(\alpha, \beta)\in \mathcal{S}_{bp}^{[K]}$ if and only if $(\beta, \alpha)\in \mathcal{S}_{bp}^{[K]}$.
\item Let $L\in [K]$ and $(\alpha, \beta)\in \mathcal{S}_{bp}^{[K]}$. Then $(\alpha, \beta)\in \mathcal{S}_{bp}^L$ if and only if $(\beta, \alpha)\in \mathcal{S}_{bp}^{{\bm}_\alpha L}$.
  \item If $(\alpha, \beta), (\alpha', \beta')\in \mathcal{S}_{bp}^{L}$ are distinct pairs,  then
  $\mathcal{D}_{\alpha}\cap \mathcal{D}_{\alpha'}=\emptyset$.
 \end{enumerate}
 \end{lemma}

 \begin{proof}
 Properties (1)--(4) are immediate. It only remains to show that the property~(5) holds.

Assume that $\mathcal{D}_{\alpha}\cap \mathcal{D}_{\alpha'}\not=\emptyset$. Then there is at least one $(n-1)$-simplex $f\in \mathcal{D}_{\alpha}\cap \mathcal{D}_{\alpha'}$ such that
\begin{enumerate}
\item[(a)] $\alpha\subseteq f$ and $\alpha'\subseteq f$ (since $\alpha=\bigcap_{I\in \mathcal{D}_{\alpha}}I$ and $\alpha'=\bigcap_{J\in \mathcal{D}_{\alpha'}}J$); and furthermore
\item[(b)] there are two vertices $v, w\in \beta$ and two vertices $v', w'\in \beta'$ such that
\[
f\,=\,(\alpha\cup\beta)\setminus\{v, w\}\,=\,(\alpha'\cup\beta')\setminus\{v', w'\}.
\]
\end{enumerate}
Since $(\alpha, \beta)$ and $(\alpha', \beta')$ are different, at most one of $v'$ and $w'$ is in $\beta$. Without the loss of generality, assume that $v'\not\in \beta$. Since $F=(\alpha'\cup\beta')\setminus\{ w'\}$ is an $n$-simplex of $L$ and $\alpha\subset f\subset F$, we have that $F\setminus\alpha$ is a simplex in  $\Link_L\alpha=\partial \beta$, so $v'\in F$ is a $0$-simplex of $\partial \beta$, that is, $v'\in \beta$. This is a contradiction. Thus $\mathcal{D}_{\alpha}\cap \mathcal{D}_{\alpha'}=\emptyset$ as desired.
 \end{proof}

\begin{example}\label{pairs}
Let $K$ be a triangulation with five vertices of the $2$-sphere, as shown in Figure~\ref{fig:2-sph}.
\setlength{\unitlength}{1184sp}%
\begingroup\makeatletter\ifx\SetFigFont\undefined%
\gdef\SetFigFont#1#2#3#4#5{%
  \reset@font\fontsize{#1}{#2pt}%
  \fontfamily{#3}\fontseries{#4}\fontshape{#5}%
  \selectfont}%
\fi\endgroup%
\begin{figure}
\begin{picture}(4044,3980)(1051,-4319)
\thicklines
\put(3301,-4261){\line(-3, 4){1800}}
{\color[rgb]{0,0,0}\multiput(1501,-1861)(400.00000,0.00000){8}{\line( 1, 0){200.000}}
}%
\put(3001,-361){\line(-1,-1){1500}}
\put(1501,-1861){\line( 3,-1){1800}}
\put(3301,-2461){\line( 2, 1){1200}}
\put(4501,-1861){\line(-1, 1){1500}}
\multiput(3001,-361)(4.58748,-27.52489){77}{\makebox(22.2222,33.3333){\SetFigFont{7}{8.4}{\rmdefault}{\mddefault}{\updefault}.}}
\put(3301,-2461){\line( 0,-1){1800}}
\put(3301,-4261){\line( 1, 2){1200}}
\put(4726,-1936){\makebox(0,0)[lb]{\smash{{\SetFigFont{5}{6.0}{\rmdefault}{\mddefault}{\updefault}$3$}}}}
\put(3451,-511){\makebox(0,0)[lb]{\smash{{\SetFigFont{5}{6.0}{\rmdefault}{\mddefault}{\updefault}$4$}}}}
\put(3601,-4261){\makebox(0,0)[lb]{\smash{{\SetFigFont{5}{6.0}{\rmdefault}{\mddefault}{\updefault}$5$}}}}
\put(1051,-1936){\makebox(0,0)[lb]{\smash{{\SetFigFont{5}{6.0}{\rmdefault}{\mddefault}{\updefault}$1$}}}}
\put(3451,-2686){\makebox(0,0)[lb]{\smash{{\SetFigFont{5}{6.0}{\rmdefault}{\mddefault}{\updefault}$2$}}}}
\end{picture}%
\caption{Triangulation of the 2-sphere with five vertices}
\label{fig:2-sph}
\centering
\end{figure}
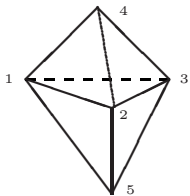
Then there are ten elements in $[K]$ and thirty bistellar pairs in $\mathcal{S}_{bp}^{[K]}$, which are described with their bistellar moves in Figure~\ref{fig:bm}.
\setlength{\unitlength}{1105sp}%
\begingroup\makeatletter\ifx\SetFigFont\undefined%
\gdef\SetFigFont#1#2#3#4#5{%
  \reset@font\fontsize{#1}{#2pt}%
  \fontfamily{#3}\fontseries{#4}\fontshape{#5}%
  \selectfont}%
\fi\endgroup%
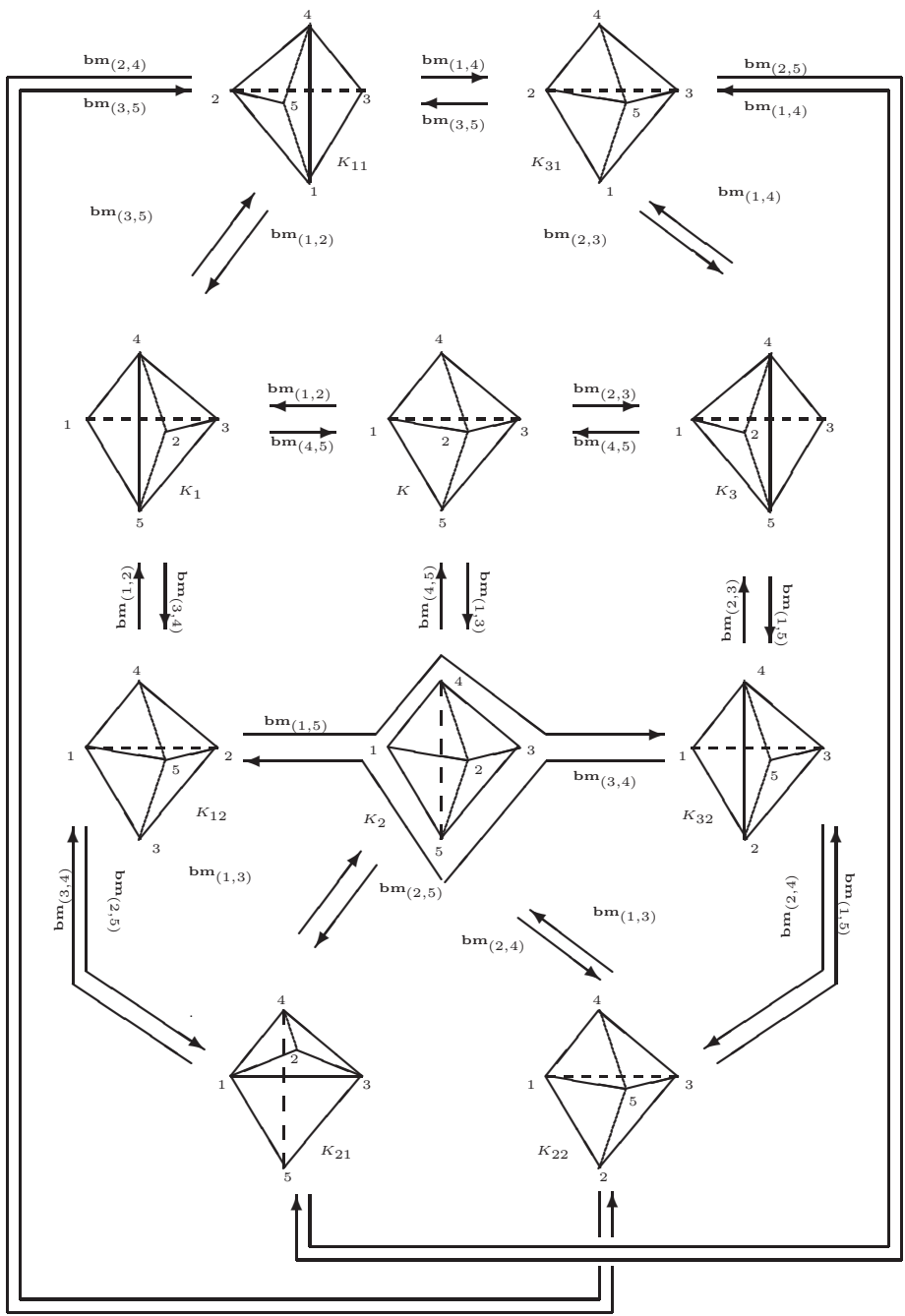
\begin{figure}
\begin{picture}(20444,29728)(-3621,-10283)
\thicklines
\put(6301,11639){\line(-4,-5){1200}}
\put(5101,10139){\line( 3,-5){1244.118}}
\put(6301,8039){\line( 5, 6){1770.492}}
\put(8101,10139){\line(-6, 5){1800}}
\multiput(13797,8052)(-9.52381,28.57143){64}{\makebox(23.8095,35.7143){\SetFigFont{7}{8.4}{\rmdefault}{\mddefault}{\updefault}.}}
\multiput(6301,4139)(0.00000,-654.54545){6}{\line( 0,-1){327.273}}
\put(6283,4154){\line(-4,-5){1200}}
\put(5083,2654){\line( 3,-5){1244.118}}
\put(6283,554){\line( 5, 6){1770.492}}
\put(8083,2654){\line(-6, 5){1800}}
\multiput(6901,2339)(-9.52381,-28.57143){64}{\makebox(23.8095,35.7143){\SetFigFont{7}{8.4}{\rmdefault}{\mddefault}{\updefault}.}}
\put(5101,2639){\line( 6,-1){1800}}
\put(6901,2339){\line( 4, 1){1200}}
\multiput(6301,4139)(9.52381,-28.57143){64}{\makebox(23.8095,35.7143){\SetFigFont{7}{8.4}{\rmdefault}{\mddefault}{\updefault}.}}
\put(6301,5339){\vector( 0, 1){1500}}
\put(6901,6839){\vector( 0,-1){1500}}
\put(9901,19139){\line(-4,-5){1200}}
\put(8701,17639){\line( 3,-5){1244.118}}
\put(9901,15539){\line( 5, 6){1770.492}}
\put(11701,17639){\line(-6, 5){1800}}
\put(8705,17663){\line( 6,-1){1800}}
\multiput(10501,17339)(-9.52381,-28.57143){64}{\makebox(23.8095,35.7143){\SetFigFont{7}{8.4}{\rmdefault}{\mddefault}{\updefault}.}}
\multiput(8701,17639)(400.00000,0.00000){8}{\line( 1, 0){200.000}}
\multiput(9901,19139)(9.52381,-28.57143){64}{\makebox(23.8095,35.7143){\SetFigFont{7}{8.4}{\rmdefault}{\mddefault}{\updefault}.}}
\put(10501,17339){\line( 4, 1){1200}}
\put(7351,17339){\vector(-1, 0){1500}}
\put(-1795,2663){\line( 6,-1){1800}}
\multiput(-1799,2639)(400.00000,0.00000){8}{\line( 1, 0){200.000}}
\multiput(-599,4139)(9.52381,-28.57143){64}{\makebox(23.8095,35.7143){\SetFigFont{7}{8.4}{\rmdefault}{\mddefault}{\updefault}.}}
\put(  1,2339){\line( 4, 1){1200}}
\multiput(  1,2339)(-9.52381,-28.57143){64}{\makebox(23.8095,35.7143){\SetFigFont{7}{8.4}{\rmdefault}{\mddefault}{\updefault}.}}
\put(-599,4139){\line(-4,-5){1200}}
\put(-1799,2639){\line( 3,-5){1244.118}}
\put(-599,539){\line( 5, 6){1770.492}}
\put(1201,2639){\line(-6, 5){1800}}
\put(9885,-3370){\line(-4,-5){1200}}
\put(8685,-4870){\line( 3,-5){1244.118}}
\put(9885,-6970){\line( 5, 6){1770.492}}
\put(11685,-4870){\line(-6, 5){1800}}
\put(2701,-3361){\line(-4,-5){1200}}
\put(1501,-4861){\line( 3,-5){1244.118}}
\put(2701,-6961){\line( 5, 6){1770.492}}
\put(4501,-4861){\line(-6, 5){1800}}
\multiput(8701,-4861)(400.00000,0.00000){8}{\line( 1, 0){200.000}}
\multiput(10501,-5161)(-9.52381,-28.57143){64}{\makebox(23.8095,35.7143){\SetFigFont{7}{8.4}{\rmdefault}{\mddefault}{\updefault}.}}
\multiput(9901,-3361)(9.52381,-28.57143){64}{\makebox(23.8095,35.7143){\SetFigFont{7}{8.4}{\rmdefault}{\mddefault}{\updefault}.}}
\put(10501,-5161){\line( 4, 1){1200}}
\put(8701,-4861){\line( 6,-1){1800}}
\multiput(2701,-3361)(0.00000,-654.54545){6}{\line( 0,-1){327.273}}
\put(1501,-4861){\line( 1, 0){3000}}
\put(3001,-4261){\line( 5,-2){1500}}
\multiput(3001,-4261)(-9.67742,29.03226){32}{\makebox(23.8095,35.7143){\SetFigFont{7}{8.4}{\rmdefault}{\mddefault}{\updefault}.}}
\put(3001,-4261){\line(-5,-2){1500}}
\put(-599,5339){\vector( 0, 1){1500}}
\put(  1,6839){\vector( 0,-1){1500}}
\put(3304,19116){\line( 4,-5){1200}}
\put(4504,17616){\line(-3,-5){1244.118}}
\put(3304,15516){\line(-5, 6){1770.492}}
\put(1504,17616){\line( 6, 5){1800}}
\put(3301,19139){\line( 0,-1){3600}}
\multiput(3297,15552)(-9.52381,28.57143){64}{\makebox(23.8095,35.7143){\SetFigFont{7}{8.4}{\rmdefault}{\mddefault}{\updefault}.}}
\multiput(3301,19139)(-9.52381,-28.57143){64}{\makebox(23.8095,35.7143){\SetFigFont{7}{8.4}{\rmdefault}{\mddefault}{\updefault}.}}
\put(2701,17339){\line(-4, 1){1200}}
\multiput(1576,17639)(400.00000,0.00000){8}{\line( 1, 0){200.000}}
\put(13228,4129){\line(-4,-5){1200}}
\put(12028,2629){\line( 3,-5){1244.118}}
\put(13228,529){\line( 5, 6){1770.492}}
\put(15028,2629){\line(-6, 5){1800}}
\put(13201,4139){\line( 0,-1){3600}}
\multiput(13801,2339)(-9.52381,-28.57143){64}{\makebox(23.8095,35.7143){\SetFigFont{7}{8.4}{\rmdefault}{\mddefault}{\updefault}.}}
\multiput(13201,4139)(9.52381,-28.57143){64}{\makebox(23.8095,35.7143){\SetFigFont{7}{8.4}{\rmdefault}{\mddefault}{\updefault}.}}
\put(13801,2339){\line( 4, 1){1200}}
\multiput(12001,2639)(400.00000,0.00000){8}{\line( 1, 0){200.000}}
\multiput(6301,11639)(9.52381,-28.57143){64}{\makebox(23.8095,35.7143){\SetFigFont{7}{8.4}{\rmdefault}{\mddefault}{\updefault}.}}
\put(6901,9839){\line( 4, 1){1200}}
\multiput(6901,9839)(-9.52381,-28.57143){64}{\makebox(23.8095,35.7143){\SetFigFont{7}{8.4}{\rmdefault}{\mddefault}{\updefault}.}}
\put(13801,6539){\vector( 0,-1){1500}}
\put(13201,5039){\vector( 0, 1){1500}}
\put(5101,10139){\line( 6,-1){1800}}
\multiput(5101,10139)(400.00000,0.00000){8}{\line( 1, 0){200.000}}
\put(-581,11624){\line(-4,-5){1200}}
\put(-1781,10124){\line( 3,-5){1244.118}}
\put(-581,8024){\line( 5, 6){1770.492}}
\put(1219,10124){\line(-6, 5){1800}}
\multiput(-577,11646)(9.52381,-28.57143){64}{\makebox(23.8095,35.7143){\SetFigFont{7}{8.4}{\rmdefault}{\mddefault}{\updefault}.}}
\put( 23,9846){\line( 4, 1){1200}}
\multiput(  1,9839)(-9.52381,-28.57143){64}{\makebox(23.8095,35.7143){\SetFigFont{7}{8.4}{\rmdefault}{\mddefault}{\updefault}.}}
\multiput(-1799,10139)(400.00000,0.00000){8}{\line( 1, 0){200.000}}
\put(-599,11639){\line( 0,-1){3600}}
\put(13801,11639){\line( 0,-1){3600}}
\put(3901,10439){\vector(-1, 0){1500}}
\put(2401,9839){\vector( 1, 0){1500}}
\put(9301,10439){\vector( 1, 0){1500}}
\put(10801,9839){\vector(-1, 0){1500}}
\put(13809,11610){\line( 4,-5){1200}}
\put(15009,10110){\line(-3,-5){1244.118}}
\put(13809,8010){\line(-5, 6){1770.492}}
\put(12009,10110){\line( 6, 5){1800}}
\put(5851,17939){\vector( 1, 0){1500}}
\multiput(12076,10139)(400.00000,0.00000){8}{\line( 1, 0){200.000}}
\multiput(13823,11630)(-9.52381,-28.57143){64}{\makebox(23.8095,35.7143){\SetFigFont{7}{8.4}{\rmdefault}{\mddefault}{\updefault}.}}
\put(13223,9830){\line(-4, 1){1200}}
\put(1351,2339){\makebox(0,0)[lb]{\smash{{\SetFigFont{5}{6.0}{\rmdefault}{\mddefault}{\updefault}$2$}}}}
\put(13801,1889){\makebox(0,0)[lb]{\smash{{\SetFigFont{5}{6.0}{\rmdefault}{\mddefault}{\updefault}$5$}}}}
\put(13351,239){\makebox(0,0)[lb]{\smash{{\SetFigFont{5}{6.0}{\rmdefault}{\mddefault}{\updefault}$2$}}}}
\put(-299,239){\makebox(0,0)[lb]{\smash{{\SetFigFont{5}{6.0}{\rmdefault}{\mddefault}{\updefault}$3$}}}}
\put(10576,-5536){\makebox(0,0)[lb]{\smash{{\SetFigFont{5}{6.0}{\rmdefault}{\mddefault}{\updefault}$5$}}}}
\put(8251,-5161){\makebox(0,0)[lb]{\smash{{\SetFigFont{5}{6.0}{\rmdefault}{\mddefault}{\updefault}$1$}}}}
\put(11776,989){\makebox(0,0)[lb]{\smash{{\SetFigFont{5}{6.0}{\rmdefault}{\mddefault}{\updefault}$K_{32}$}}}}
\put(676,1064){\makebox(0,0)[lb]{\smash{{\SetFigFont{5}{6.0}{\rmdefault}{\mddefault}{\updefault}$K_{12}$}}}}
\put(3526,-6661){\makebox(0,0)[lb]{\smash{{\SetFigFont{5}{6.0}{\rmdefault}{\mddefault}{\updefault}$K_{21}$}}}}
\put(8476,-6661){\makebox(0,0)[lb]{\smash{{\SetFigFont{5}{6.0}{\rmdefault}{\mddefault}{\updefault}$K_{22}$}}}}
\put(301,8414){\makebox(0,0)[lb]{\smash{{\SetFigFont{5}{6.0}{\rmdefault}{\mddefault}{\updefault}$K_1$}}}}
\put(3901,15914){\makebox(0,0)[lb]{\smash{{\SetFigFont{5}{6.0}{\rmdefault}{\mddefault}{\updefault}$K_{11}$}}}}
\put(12526,8414){\makebox(0,0)[lb]{\smash{{\SetFigFont{5}{6.0}{\rmdefault}{\mddefault}{\updefault}$K_3$}}}}
\put(8326,15914){\makebox(0,0)[lb]{\smash{{\SetFigFont{5}{6.0}{\rmdefault}{\mddefault}{\updefault}$K_{31}$}}}}
\put(6601,4064){\makebox(0,0)[lb]{\smash{{\SetFigFont{5}{6.0}{\rmdefault}{\mddefault}{\updefault}$4$}}}}
\put(6151,164){\makebox(0,0)[lb]{\smash{{\SetFigFont{5}{6.0}{\rmdefault}{\mddefault}{\updefault}$5$}}}}
\put(4651,2489){\makebox(0,0)[lb]{\smash{{\SetFigFont{5}{6.0}{\rmdefault}{\mddefault}{\updefault}$1$}}}}
\put(8251,2489){\makebox(0,0)[lb]{\smash{{\SetFigFont{5}{6.0}{\rmdefault}{\mddefault}{\updefault}$3$}}}}
\put(4501,914){\makebox(0,0)[lb]{\smash{{\SetFigFont{5}{6.0}{\rmdefault}{\mddefault}{\updefault}$K_2$}}}}
\put(-1874,18239){\makebox(0,0)[lb]{\smash{{\SetFigFont{5}{6.0}{\rmdefault}{\mddefault}{\updefault}${\bf bm}_{(2,4)}$}}}}
\put(-1874,17264){\makebox(0,0)[lb]{\smash{{\SetFigFont{5}{6.0}{\rmdefault}{\mddefault}{\updefault}${\bf bm}_{(3,5)}$}}}}
\put(5851,18239){\makebox(0,0)[lb]{\smash{{\SetFigFont{5}{6.0}{\rmdefault}{\mddefault}{\updefault}${\bf bm}_{(1,4)}$}}}}
\put(5851,16889){\makebox(0,0)[lb]{\smash{{\SetFigFont{5}{6.0}{\rmdefault}{\mddefault}{\updefault}${\bf bm}_{(3,5)}$}}}}
\put(13201,18164){\makebox(0,0)[lb]{\smash{{\SetFigFont{5}{6.0}{\rmdefault}{\mddefault}{\updefault}${\bf bm}_{(2,5)}$}}}}
\put(13201,17189){\makebox(0,0)[lb]{\smash{{\SetFigFont{5}{6.0}{\rmdefault}{\mddefault}{\updefault}${\bf bm}_{(1,4)}$}}}}
\put(-1724,14789){\makebox(0,0)[lb]{\smash{{\SetFigFont{5}{6.0}{\rmdefault}{\mddefault}{\updefault}${\bf bm}_{(3,5)}$}}}}
\put(2401,14264){\makebox(0,0)[lb]{\smash{{\SetFigFont{5}{6.0}{\rmdefault}{\mddefault}{\updefault}${\bf bm}_{(1,2)}$}}}}
\put(2326,10739){\makebox(0,0)[lb]{\smash{{\SetFigFont{5}{6.0}{\rmdefault}{\mddefault}{\updefault}${\bf bm}_{(1,2)}$}}}}
\put(2401,9539){\makebox(0,0)[lb]{\smash{{\SetFigFont{5}{6.0}{\rmdefault}{\mddefault}{\updefault}${\bf bm}_{(4,5)}$}}}}
\put(9301,10739){\makebox(0,0)[lb]{\smash{{\SetFigFont{5}{6.0}{\rmdefault}{\mddefault}{\updefault}${\bf bm}_{(2,3)}$}}}}
\put(9301,9539){\makebox(0,0)[lb]{\smash{{\SetFigFont{5}{6.0}{\rmdefault}{\mddefault}{\updefault}${\bf bm}_{(4,5)}$}}}}
\put(-899,5339){\rotatebox{90.0}{\makebox(0,0)[lb]{\smash{{\SetFigFont{5}{6.0}{\rmdefault}{\mddefault}{\updefault}${\bf bm}_{(1,2)}$}}}}}
\put(301,6689){\rotatebox{270.0}{\makebox(0,0)[lb]{\smash{{\SetFigFont{5}{6.0}{\rmdefault}{\mddefault}{\updefault}${\bf bm}_{(3,4)}$}}}}}
\put(6001,5339){\rotatebox{90.0}{\makebox(0,0)[lb]{\smash{{\SetFigFont{5}{6.0}{\rmdefault}{\mddefault}{\updefault}${\bf bm}_{(4,5)}$}}}}}
\put(7201,6689){\rotatebox{270.0}{\makebox(0,0)[lb]{\smash{{\SetFigFont{5}{6.0}{\rmdefault}{\mddefault}{\updefault}${\bf bm}_{(1,3)}$}}}}}
\put(12901,5039){\rotatebox{90.0}{\makebox(0,0)[lb]{\smash{{\SetFigFont{5}{6.0}{\rmdefault}{\mddefault}{\updefault}${\bf bm}_{(2,3)}$}}}}}
\put(14101,6389){\rotatebox{270.0}{\makebox(0,0)[lb]{\smash{{\SetFigFont{5}{6.0}{\rmdefault}{\mddefault}{\updefault}${\bf bm}_{(1,5)}$}}}}}
\put(2251,3164){\makebox(0,0)[lb]{\smash{{\SetFigFont{5}{6.0}{\rmdefault}{\mddefault}{\updefault}${\bf bm}_{(1,5)}$}}}}
\put(9301,1889){\makebox(0,0)[lb]{\smash{{\SetFigFont{5}{6.0}{\rmdefault}{\mddefault}{\updefault}${\bf bm}_{(3,4)}$}}}}
\put(-2324,-1561){\rotatebox{90.0}{\makebox(0,0)[lb]{\smash{{\SetFigFont{5}{6.0}{\rmdefault}{\mddefault}{\updefault}${\bf bm}_{(3,4)}$}}}}}
\put(4876,-661){\makebox(0,0)[lb]{\smash{{\SetFigFont{5}{6.0}{\rmdefault}{\mddefault}{\updefault}${\bf bm}_{(2,5)}$}}}}
\put(6751,-1861){\makebox(0,0)[lb]{\smash{{\SetFigFont{5}{6.0}{\rmdefault}{\mddefault}{\updefault}${\bf bm}_{(2,4)}$}}}}
\put(9751,-1186){\makebox(0,0)[lb]{\smash{{\SetFigFont{5}{6.0}{\rmdefault}{\mddefault}{\updefault}${\bf bm}_{(1,3)}$}}}}
\put(15526,-211){\rotatebox{270.0}{\makebox(0,0)[lb]{\smash{{\SetFigFont{5}{6.0}{\rmdefault}{\mddefault}{\updefault}${\bf bm}_{(1,5)}$}}}}}
\put(-1124,-136){\rotatebox{270.0}{\makebox(0,0)[lb]{\smash{{\SetFigFont{5}{6.0}{\rmdefault}{\mddefault}{\updefault}${\bf bm}_{(2,5)}$}}}}}
\put(526,-286){\makebox(0,0)[lb]{\smash{{\SetFigFont{5}{6.0}{\rmdefault}{\mddefault}{\updefault}${\bf bm}_{(1,3)}$}}}}
\put(14251,-1711){\rotatebox{90.0}{\makebox(0,0)[lb]{\smash{{\SetFigFont{5}{6.0}{\rmdefault}{\mddefault}{\updefault}${\bf bm}_{(2,4)}$}}}}}
\put(12601,15239){\makebox(0,0)[lb]{\smash{{\SetFigFont{5}{6.0}{\rmdefault}{\mddefault}{\updefault}${\bf bm}_{(1,4)}$}}}}
\put(8626,14264){\makebox(0,0)[lb]{\smash{{\SetFigFont{5}{6.0}{\rmdefault}{\mddefault}{\updefault}${\bf bm}_{(2,3)}$}}}}
\put(2337,14893){\vector(-3,-4){1404}}
\put(10847,14876){\vector( 4,-3){1872}}
\put(3060,-1610){\vector( 3, 4){1404}}
\put(8059,-1251){\vector( 4,-3){1872}}
\thinlines
{\color[rgb]{0,0,0}\put(601,-3511){\makebox(5.9524,41.6667){\SetFigFont{5}{6}{\rmdefault}{\mddefault}{\updefault}.}}
}%
\thicklines
\put(-2099,-2761){\vector( 0, 1){3600}}
\put(-2099,-2761){\line( 3,-2){2700}}
\put(-1799,839){\line( 0,-1){3300}}
\put(-1799,-2461){\vector( 3,-2){2700}}
\put(15001,866){\line( 0,-1){3300}}
\put(15001,-2434){\vector(-3,-2){2700}}
\put(15301,-2761){\vector( 0, 1){3600}}
\put(15301,-2761){\line(-3,-2){2700}}
\put(12601,17939){\line( 1, 0){4200}}
\put(16801,17939){\line( 0,-1){27000}}
\put(16801,-9061){\line(-1, 0){13800}}
\put(3001,-9061){\vector( 0, 1){1500}}
\put(16501,17639){\vector(-1, 0){3900}}
\put(16501,17639){\line( 0,-1){26400}}
\put(16501,-8761){\line(-1, 0){13200}}
\put(3301,-8761){\line( 0, 1){1200}}
\put(9901,-9211){\line( 0,-1){750}}
\put(9901,-9961){\line(-1, 0){13200}}
\put(-3299,-9961){\line( 0, 1){27600}}
\put(-3299,17639){\vector( 1, 0){3900}}
\put(601,17939){\line(-1, 0){4200}}
\put(-3599,17939){\line( 0,-1){28200}}
\put(-3599,-10261){\line( 1, 0){13800}}
\put(10201,-10261){\line( 0, 1){1050}}
\put(10201,-9211){\line( 0,-1){ 75}}
\put(10201,-8536){\vector( 0, 1){1050}}
\put(9901,-7486){\line( 0,-1){1050}}
\put(4833,-38){\vector(-3,-4){1404}}
\put(10203,-2475){\vector(-4, 3){1872}}
\put(631,13415){\vector( 3, 4){1404}}
\put(12932,13701){\vector(-4, 3){1872}}
\put(1801,2939){\line( 1, 0){3000}}
\put(4801,2939){\line( 5, 6){1500}}
\put(6301,4739){\line( 4,-3){2400}}
\put(8701,2939){\vector( 1, 0){2700}}
\put(11401,2339){\line(-1, 0){2700}}
\put(8701,2339){\line(-5,-6){2311.475}}
\put(6301,-361){\line(-2, 3){1800}}
\put(4501,2339){\vector(-1, 0){2700}}
\put(6151,11864){\makebox(0,0)[lb]{\smash{{\SetFigFont{5}{6.0}{\rmdefault}{\mddefault}{\updefault}$4$}}}}
\put(4651,9764){\makebox(0,0)[lb]{\smash{{\SetFigFont{5}{6.0}{\rmdefault}{\mddefault}{\updefault}$1$}}}}
\put(6301,9539){\makebox(0,0)[lb]{\smash{{\SetFigFont{5}{6.0}{\rmdefault}{\mddefault}{\updefault}$2$}}}}
\put(8101,9764){\makebox(0,0)[lb]{\smash{{\SetFigFont{5}{6.0}{\rmdefault}{\mddefault}{\updefault}$3$}}}}
\put(6226,7664){\makebox(0,0)[lb]{\smash{{\SetFigFont{5}{6.0}{\rmdefault}{\mddefault}{\updefault}$5$}}}}
\put(13651,11789){\makebox(0,0)[lb]{\smash{{\SetFigFont{5}{6.0}{\rmdefault}{\mddefault}{\updefault}$4$}}}}
\put(15076,9839){\makebox(0,0)[lb]{\smash{{\SetFigFont{5}{6.0}{\rmdefault}{\mddefault}{\updefault}$3$}}}}
\put(11626,9764){\makebox(0,0)[lb]{\smash{{\SetFigFont{5}{6.0}{\rmdefault}{\mddefault}{\updefault}$1$}}}}
\put(13351,9689){\makebox(0,0)[lb]{\smash{{\SetFigFont{5}{6.0}{\rmdefault}{\mddefault}{\updefault}$2$}}}}
\put(13726,7664){\makebox(0,0)[lb]{\smash{{\SetFigFont{5}{6.0}{\rmdefault}{\mddefault}{\updefault}$5$}}}}
\put(-749,11864){\makebox(0,0)[lb]{\smash{{\SetFigFont{5}{6.0}{\rmdefault}{\mddefault}{\updefault}$4$}}}}
\put(-2324,9914){\makebox(0,0)[lb]{\smash{{\SetFigFont{5}{6.0}{\rmdefault}{\mddefault}{\updefault}$1$}}}}
\put(1276,9839){\makebox(0,0)[lb]{\smash{{\SetFigFont{5}{6.0}{\rmdefault}{\mddefault}{\updefault}$3$}}}}
\put(151,9539){\makebox(0,0)[lb]{\smash{{\SetFigFont{5}{6.0}{\rmdefault}{\mddefault}{\updefault}$2$}}}}
\put(-674,7664){\makebox(0,0)[lb]{\smash{{\SetFigFont{5}{6.0}{\rmdefault}{\mddefault}{\updefault}$5$}}}}
\put(5251,8414){\makebox(0,0)[lb]{\smash{{\SetFigFont{5}{6.0}{\rmdefault}{\mddefault}{\updefault}$K$}}}}
\put(7051,2039){\makebox(0,0)[lb]{\smash{{\SetFigFont{5}{6.0}{\rmdefault}{\mddefault}{\updefault}$2$}}}}
\put(4501,17339){\makebox(0,0)[lb]{\smash{{\SetFigFont{5}{6.0}{\rmdefault}{\mddefault}{\updefault}$3$}}}}
\put(9751,19289){\makebox(0,0)[lb]{\smash{{\SetFigFont{5}{6.0}{\rmdefault}{\mddefault}{\updefault}$4$}}}}
\put(-749,4289){\makebox(0,0)[lb]{\smash{{\SetFigFont{5}{6.0}{\rmdefault}{\mddefault}{\updefault}$4$}}}}
\put(13501,4289){\makebox(0,0)[lb]{\smash{{\SetFigFont{5}{6.0}{\rmdefault}{\mddefault}{\updefault}$4$}}}}
\put(11701,2339){\makebox(0,0)[lb]{\smash{{\SetFigFont{5}{6.0}{\rmdefault}{\mddefault}{\updefault}$1$}}}}
\put(15001,2339){\makebox(0,0)[lb]{\smash{{\SetFigFont{5}{6.0}{\rmdefault}{\mddefault}{\updefault}$3$}}}}
\put(9751,-3211){\makebox(0,0)[lb]{\smash{{\SetFigFont{5}{6.0}{\rmdefault}{\mddefault}{\updefault}$4$}}}}
\put(11851,-5161){\makebox(0,0)[lb]{\smash{{\SetFigFont{5}{6.0}{\rmdefault}{\mddefault}{\updefault}$3$}}}}
\put(9901,-7261){\makebox(0,0)[lb]{\smash{{\SetFigFont{5}{6.0}{\rmdefault}{\mddefault}{\updefault}$2$}}}}
\put(2551,-3211){\makebox(0,0)[lb]{\smash{{\SetFigFont{5}{6.0}{\rmdefault}{\mddefault}{\updefault}$4$}}}}
\put(2851,-4561){\makebox(0,0)[lb]{\smash{{\SetFigFont{5}{6.0}{\rmdefault}{\mddefault}{\updefault}$2$}}}}
\put(1201,-5161){\makebox(0,0)[lb]{\smash{{\SetFigFont{5}{6.0}{\rmdefault}{\mddefault}{\updefault}$1$}}}}
\put(4501,-5161){\makebox(0,0)[lb]{\smash{{\SetFigFont{5}{6.0}{\rmdefault}{\mddefault}{\updefault}$3$}}}}
\put(2701,-7261){\makebox(0,0)[lb]{\smash{{\SetFigFont{5}{6.0}{\rmdefault}{\mddefault}{\updefault}$5$}}}}
\put(-2249,2339){\makebox(0,0)[lb]{\smash{{\SetFigFont{5}{6.0}{\rmdefault}{\mddefault}{\updefault}$1$}}}}
\put(2851,17189){\makebox(0,0)[lb]{\smash{{\SetFigFont{5}{6.0}{\rmdefault}{\mddefault}{\updefault}$5$}}}}
\put(3151,19289){\makebox(0,0)[lb]{\smash{{\SetFigFont{5}{6.0}{\rmdefault}{\mddefault}{\updefault}$4$}}}}
\put(8251,17489){\makebox(0,0)[lb]{\smash{{\SetFigFont{5}{6.0}{\rmdefault}{\mddefault}{\updefault}$2$}}}}
\put(10651,17039){\makebox(0,0)[lb]{\smash{{\SetFigFont{5}{6.0}{\rmdefault}{\mddefault}{\updefault}$5$}}}}
\put(11851,17489){\makebox(0,0)[lb]{\smash{{\SetFigFont{5}{6.0}{\rmdefault}{\mddefault}{\updefault}$3$}}}}
\put(10051,15239){\makebox(0,0)[lb]{\smash{{\SetFigFont{5}{6.0}{\rmdefault}{\mddefault}{\updefault}$1$}}}}
\put(3301,15239){\makebox(0,0)[lb]{\smash{{\SetFigFont{5}{6.0}{\rmdefault}{\mddefault}{\updefault}$1$}}}}
\put(1051,17339){\makebox(0,0)[lb]{\smash{{\SetFigFont{5}{6.0}{\rmdefault}{\mddefault}{\updefault}$2$}}}}
\put(151,2039){\makebox(0,0)[lb]{\smash{{\SetFigFont{5}{6.0}{\rmdefault}{\mddefault}{\updefault}$5$}}}}
\end{picture}%
\caption{Bistellar pairs and moves in $[K]$.}
\label{fig:bm}
\centering
\end{figure}
 \end{example}

\begin{definition}
The \emph{bistellar exchange graph} $\mathbb{G}_{[K]}$ of $[K]$ is defined as follows:
\begin{enumerate}
\item The vertex set of $\mathbb{G}_{[K]}$ is $[K]$.
\item Each edge corresponds to a bistellar $h$-move and its inverse. In other words, for $(\alpha, \beta)\in\mathcal{S}_{bp}^{[K]}$, if $(\alpha, \beta)$ is a bistellar pair in $L\in [K]$, then both $(\alpha, \beta)$ and $(\beta, \alpha)$ define an edge between $L$ and ${\bm}_\alpha L$.
 \end{enumerate}
 \end{definition}

The bistellar exchange graph $\mathbb{G}_{[K]}$ of $[K]$ indicates how all triangulated manifolds in $[K]$ exchange amongst themselves by performing bistellar $h$-moves.

\begin{example}
The bistellar exchange graph of $[K]$ in Example~\ref{pairs} is a trivalent graph with ten vertices shown in Figure~\ref{fig:graph}.
\setlength{\unitlength}{1184sp}%
\begingroup\makeatletter\ifx\SetFigFont\undefined%
\gdef\SetFigFont#1#2#3#4#5{%
  \reset@font\fontsize{#1}{#2pt}%
  \fontfamily{#3}\fontseries{#4}\fontshape{#5}%
  \selectfont}%
\fi\endgroup%
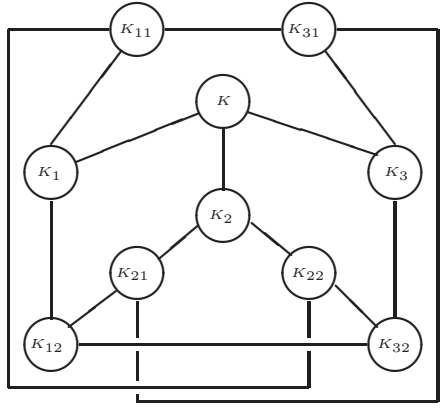
\begin{figure}
\begin{picture}(9044,8382)(3879,9217)
\put(6226,16964){\makebox(0,0)[lb]{\smash{{\SetFigFont{5}{6.0}{\rmdefault}{\mddefault}{\updefault}$K_{11}$}}}}
\thicklines
\put(4801,14039){\circle{1092}}
\put(4801,10439){\circle{1092}}
\put(12001,10439){\circle{1092}}
\put(6601,11939){\circle{1092}}
\put(8401,15539){\circle{1092}}
\put(8401,13139){\circle{1092}}
\put(10201,11939){\circle{1092}}
\put(10201,17039){\circle{1092}}
\put(6601,17039){\circle{1092}}
\put(12001,13439){\line( 0,-1){2400}}
\put(4801,13439){\line( 0,-1){2475}}
\put(6001,17039){\line(-1, 0){2100}}
\put(3901,17039){\line( 0,-1){7500}}
\put(3901,9539){\line( 1, 0){6300}}
\put(10201,9539){\line( 0, 1){750}}
\put(7201,17039){\line( 1, 0){2400}}
\put(12001,14639){\line(-3, 4){1449}}
\put(10801,17039){\line( 1, 0){2100}}
\put(12901,17039){\line( 0,-1){7800}}
\put(12901,9239){\line(-1, 0){6300}}
\put(6601,9239){\line( 0, 1){150}}
\put(8926,15314){\line( 3,-1){2700}}
\put(5326,14264){\line( 5, 2){2560.345}}
\put(4801,14639){\line( 3, 4){1476}}
\put(8401,14939){\line( 0,-1){1275}}
\put(9826,12314){\line(-6, 5){818.852}}
\put(7051,12239){\line( 6, 5){818.852}}
\put(5176,10814){\line( 4, 3){984}}
\put(11626,10814){\line(-1, 1){862.500}}
\put(5401,10439){\line( 1, 0){6000}}
\put(6601,10589){\line( 0, 1){750}}
\put(10201,10589){\line( 0, 1){750}}
\put(6601,9689){\line( 0, 1){600}}
\put(4501,13964){\makebox(0,0)[lb]{\smash{{\SetFigFont{5}{6.0}{\rmdefault}{\mddefault}{\updefault}$K_1$}}}}
\put(11776,13964){\makebox(0,0)[lb]{\smash{{\SetFigFont{5}{6.0}{\rmdefault}{\mddefault}{\updefault}$K_3$}}}}
\put(11626,10364){\makebox(0,0)[lb]{\smash{{\SetFigFont{5}{6.0}{\rmdefault}{\mddefault}{\updefault}$K_{32}$}}}}
\put(4351,10364){\makebox(0,0)[lb]{\smash{{\SetFigFont{5}{6.0}{\rmdefault}{\mddefault}{\updefault}$K_{12}$}}}}
\put(6151,11864){\makebox(0,0)[lb]{\smash{{\SetFigFont{5}{6.0}{\rmdefault}{\mddefault}{\updefault}$K_{21}$}}}}
\put(8251,15464){\makebox(0,0)[lb]{\smash{{\SetFigFont{5}{6.0}{\rmdefault}{\mddefault}{\updefault}$K$}}}}
\put(8101,13064){\makebox(0,0)[lb]{\smash{{\SetFigFont{5}{6.0}{\rmdefault}{\mddefault}{\updefault}$K_2$}}}}
\put(9826,11864){\makebox(0,0)[lb]{\smash{{\SetFigFont{5}{6.0}{\rmdefault}{\mddefault}{\updefault}$K_{22}$}}}}
\put(9751,16964){\makebox(0,0)[lb]{\smash{{\SetFigFont{5}{6.0}{\rmdefault}{\mddefault}{\updefault}$K_{31}$}}}}
\put(12001,14039){\circle{1092}}
\end{picture}%
\caption{Bistellar exchange graph.}
\label{fig:graph}
\centering
\end{figure}
 \end{example}

 \begin{proposition}
 The bistellar exchange graph $\mathbb{G}_{[K]}$ is connected.
 \end{proposition}

 \begin{proof}
 This immediately follows from $[K]$ being an equivalence class.
 \end{proof}

\begin{remark}
Generally, the bistellar exchange graph $\mathbb{G}_{[K]}$ may not be regular.
\end{remark}

Now, to each vertex $L\in \mathbb{G}_{[K]}$ (that is, a triangulated manifold in $[K]$), we can associate two pieces of data: a bistellar seed $\Sigma_L=(\mathcal{X}(L), {\bf p}_L, B(L))$ and a field $\mathbb{F}\mathcal{X}(L)$. To each edge $\overline{LL'}$ in $\mathbb{G}_{[K]}$ (that is, there is a bistellar pair $(\alpha, \beta)\in \mathcal{S}_{bp}^{[K]}$ such that $(\alpha, \beta)\in \mathcal{S}_{bp}^L$ and $(\beta, \alpha)\in \mathcal{S}_{bp}^{L'}$, so $L'={\bm}_\alpha L$ and $L={\bm}_\beta L'$), we can associate the seed mutation $\Phi_{\alpha}\colon \Sigma_L\to \Sigma_{L'}$ and its inverse $\Phi_{\beta}\colon \Sigma_{L'}\to \Sigma_L$. We know from~\eqref{field-iso} that the seed mutation $\Phi_{\alpha}$ determines a $\mathbb{ZP}$-linear field isomorphism $\mathcal{L}_{(L, L')}\colon \mathbb{F}\mathcal{X}(L)\to\mathbb{F}\mathcal{X}(L')$ given by
\[
\mathcal{L}_{(L, L')}(x_f)\,=\,\varphi^{L}_{\alpha}(x_f)\,=\,
\begin{cases}
    x_f, & \text{if $f\in \mathcal{F}(L)\setminus \mathcal{F}(\Lambda_\alpha)$;}\\
    x_{\sigma(f)}, & \text{if $f\in  \mathcal{F}(\Lambda_\alpha)\setminus\mathcal{D}_{\alpha}$;}\\
    x_{\sigma(f)}=\frac{M^+_{L,\alpha,f}+M^-_{L,\alpha,f}}{x_f}, & \text{if $f\in \mathcal{D}_{\alpha}$,}
    \end{cases}
\]
and the isomorphism $\mathcal{L}_{(L', L)}$ determined by $\Phi_{\beta}$ is exactly $\mathcal{L}_{(L, L')}^{-1}$. By Corollary~\ref{field-1},  $\mathbb{F}\mathcal{X}(L)$ can be identified with $\mathbb{F}\mathcal{X}(L')$. Since $\mathbb{G}_{[K]}$ is connected, all the fields $\mathbb{F}\mathcal{X}(L)$, for $L\in [K]$, can be identified with each other, and then they can be regarded as a single field $\mathbb{F}$ such that:
 \begin{enumerate}
\item $\mathbb{F}$ contains all the elements in $\mathcal{X}_{[K]}=\bigcup _{L\in [K]}\mathcal{X}(L)$;
\item there are the exchange relations (\ref{e-relation}) in $\mathbb{F}$ for exchangeable bistellar variables
\[
\mathcal{EX}_{[K]}\,=\,\left\{ x_f\mid f\in \bigcup_{(\alpha, \beta)\in \mathcal{S}_{bp}^{[K]}}\mathcal{D}_{\alpha}\cup \mathcal{D}_{\beta} \right\}.
\]
\end{enumerate}

\begin{definition}\label{alg1}
The \emph{bistellar cluster algebra} $\mathcal{A}_{[K]}$ of $[K]$ is the subalgebra of $\mathbb{F}$, generated by all cluster variables $x_f\in \mathcal{X}_{[K]}$ with the exchange relations given by~\eqref{e-relation}.
\end{definition}

Definition~\ref{alg1} is still suitable for the case of  $|[K]|=1$. In this case, $\mathcal{A}_{[K]}$ is just a polynomial ring generated by all variables $x_f\in \mathcal{X}(K)$.

\begin{remark}
Since $\mathbb{G}_{[K]}$ is connected, all exchangeable bistellar variables $x_f\in \mathcal{EX}_{[K]}$ can be reduced to being rational functions in $\mathbb{F}\mathcal{X}(K)$ by using the exchange relations. So the bistellar cluster algebra $\mathcal{A}_{[K]}$ can  be regarded as a subalgebra of $\mathbb{F}\mathcal{X}(K)$ with $\Sigma_L=(\mathcal{X}(K), {\bf p}_K, B(K))$ as an initial bistellar seed.
\end{remark}

\begin{definition}\label{alg2}
For each $L\in [K]$, $\mathcal{A}_L$ is defined as  the subalgebra in $\mathbb{F}\mathcal{X}(L)$ with $\Sigma_L=(\mathcal{X}(L), {\bf p}_L, B(L))$ as an initial bistellar  seed.
\end{definition}

\begin{theorem}\label{main1}
Let $K, K'\in \mathbb{TM}_{2h}$ be two triangulated manifolds. If $K'$ is obtained from $K$ by a finite sequence of bistellar $h$-moves, then
$\mathcal{A}_{K}\cong \mathcal{A}_{K'}$.
\end{theorem}

\begin{proof}
This is because $[K]=[K']$.
\end{proof}


\subsection{A Simple 2-dimensional Example}\label{example1}

Let $K$ be the triangulation of the 2-sphere with five vertices in Example~\ref{pairs}. So there are ten elements in $[K]$ and thirty bistellar pairs in $\mathcal{S}_{bp}^{[K]}$, as shown in Example~\ref{pairs}.
First we see easily that
\begin{enumerate}
\item The union $\bigcup_{L\in [K]}\mathcal{F}(L)$ is equal to $\mathcal{F}_1^{2^{[5]}}$. So
$\mathcal{X}_{[K]}$ consists of ten elements.  More precisely,
\[
\mathcal{X}_{[K]}\,=\,\{x_{(1,2)}, x_{(1,3)}, x_{(1,4)}, x_{(1,5)}, x_{(2,3)}, x_{(2,4)}, x_{(2,5)}, x_{(3,4)}, x_{(3,5)}, x_{(4,5)}\}.
\]
\item The set $\mathcal{EX}_{[K]}$ of all exchangeable bistellar variables is  the same as $\mathcal{X}_{[K]}$. So $\mathcal{A}_{[K]}$ is the subalgebra of $\mathbb{F}$, generated by the ten elements in $\mathcal{X}_{[K]}$.
\end{enumerate}

To determine the structure of $\mathcal{A}_{[K]}$, let us look at the exchange relations between  exchangeable bistellar variables. Since $\mathcal{S}_{bp}^{[K]}$ contains thirty bistellar pairs, there are 15 exchange relations. These exchange relations with the relations among the ${\bf p}_L$, for $L\in [K]$, are stated as follows:
\begin{enumerate}
\item There is the following exchange relation between $\mathcal{X}(K)$ and $\mathcal{X}(K_1)$.
\begin{align*}
x_{(1,2)}x_{(4,5)}\,=\, & p_{K, x_{(1,2)}}^+x_{(1,4)}x_{(2,5)}+p_{K,x_{(1,2)}}^-x_{(1,5)}x_{(2,4)} \\
=\, & p_{K_1, x_{(4,5)}}^-x_{(1,4)}x_{(2,5)}+p_{K_1, x_{(4,5)}}^+x_{(1,5)}x_{(2,4)}
\end{align*}
with $p_{K, x_{(1,2)}}^\pm=p_{K_1, x_{(4,5)}}^\mp$.

\item There is the following exchange relation between $\mathcal{X}(K)$ and $\mathcal{X}(K_2)$.
\begin{align*}
x_{(1,3)}x_{(4,5)}\,=\, &p_{K, x_{(1,3)}}^+x_{(1,4)}x_{(3,5)}+p_{K, x_{(1,3)}}^-x_{(1,5)}x_{(3,4)}\\
=\, &p_{K_2, x_{(4,5)}}^-x_{(1,4)}x_{(3,5)}+p_{K_2, x_{(4,5)}}^+x_{(1,5)}x_{(3,4)}
\end{align*}
with $p_{K, x_{(1,3)}}^\pm=p_{K_2, x_{(4,5)}}^\mp$.

\item There is the following exchange relation between $\mathcal{X}(K)$ and $\mathcal{X}(K_3)$.
\begin{align*}
x_{(2,3)}x_{(4,5)}\,=\, & p_{K, x_{(2,3)}}^+x_{(2,4)}x_{(3,5)}+p_{K, x_{(2,3)}}^-x_{(2,5)}x_{(3,4)}\\
=\, & p_{K_3, x_{(4,5)}}^-x_{(2,4)}x_{(3,5)}+p_{K_3, x_{(4,5)}}^+x_{(2,5)}x_{(3,4)}
\end{align*}
with $p_{K, x_{(2,3)}}^\pm=p_{K_3, x_{(4,5)}}^\mp$.

\item There is the following exchange relation between $\mathcal{X}(K_1)$ and $\mathcal{X}(K_{11})$.
\begin{align*}
 x_{(1,2)}x_{(3,5)}\,=\, & p_{K_{11}, x_{(1,2)}}^+x_{(1,3)}x_{(2,5)}+p_{K_{11}, x_{(1,2)}}^-x_{(1,5)}x_{(2,3)}\\
 =\, &p_{K_{1}, x_{(3,5)}}^-x_{(1,3)}x_{(2,5)}+p_{K_{1}, x_{(3,5)}}^+x_{(1,5)}x_{(2,3)}
 \end{align*}
with  $p_{K_{11}, x_{(1,2)}}^\pm=p_{K_{1}, x_{(3,5)}}^\mp$.

\item There is the following exchange relation between $\mathcal{X}(K_1)$ and $\mathcal{X}(K_{12})$.
\begin{align*}
 x_{(1,2)}x_{(3,4)}\,=\, & p_{K_{12}, x_{(1,2)}}^+x_{(1,3)}x_{(2,4)}+p_{K_{12}, x_{(1,2)}}^-x_{(1,4)}x_{(2,3)}\\
 =\, & p_{K_{1}, x_{(3,4)}}^-x_{(1,3)}x_{(2,4)}+p_{K_{1}, x_{(3,4)}}^+x_{(1,4)}x_{(2,3)}
 \end{align*}
 with $p_{K_{12}, x_{(1,2)}}^\pm=p_{K_{1}, x_{(3,4)}}^\mp$.

\item There is the following exchange relation between $\mathcal{X}(K_2)$ and $\mathcal{X}(K_{21})$.
\begin{align*}
x_{(1,3)}x_{(2,5)}\,=\, & p_{K_{21}, x_{(1,3)}}^+x_{(1,2)}x_{(3,5)}+p_{K_{21}, x_{(1,3)}}^-x_{(1,5)}x_{(2,3)}\\
=\, & p_{K_{2}, x_{(2,5)}}^-x_{(1,2)}x_{(3,5)}+p_{K_{2}, x_{(2,5)}}^+x_{(1,5)}x_{(2,3)}
\end{align*}
with $p_{K_{21}, x_{(1,3)}}^\pm=p_{K_{2}, x_{(2,5)}}^\mp$.

\item There is the following exchange relation between $\mathcal{X}(K_2)$ and $\mathcal{X}(K_{22})$.
\begin{align*}
x_{(1,3)}x_{(2,4)}\,=\, & p_{K_{22}, x_{(1,3)}}^+x_{(1,2)}x_{(3,4)}+p_{K_{22}, x_{(1,3)}}^-x_{(1,4)}x_{(2,3)}\\
=\, & p_{K_{2}, x_{(2,4)}}^-x_{(1,2)}x_{(3,4)}+p_{K_{2}, x_{(2,4)}}^+x_{(1,4)}x_{(2,3)}
 \end{align*}
 with $p_{K_{22}, x_{(1,3)}}^\pm=p_{K_{2}, x_{(2,4)}}^\mp$.

\item There is the following exchange relation between $\mathcal{X}(K_3)$ and $\mathcal{X}(K_{31})$.
\begin{align*}
x_{(1,4)}x_{(2,3)}\,=\, & p_{K_{3}, x_{(1,4)}}^+x_{(1,2)}x_{(3,4)}+p_{K_{3}, x_{(1,4)}}^-x_{(1,3)}x_{(2,4)}\\
=\, & p_{K_{31}, x_{(2,3)}}^-x_{(1,2)}x_{(3,4)}+p_{K_{31}, x_{(2,3)}}^+x_{(1,3)}x_{(2,4)}
 \end{align*}
 with $p_{K_{31}, x_{(2,3)}}^\pm=p_{K_{3}, x_{(1,4)}}^\mp$.

\item There is the following exchange relation between $\mathcal{X}(K_3)$ and $\mathcal{X}(K_{32})$.
\begin{align*}
x_{(1,5)}x_{(2,3)}\,=\, & p_{K_{3}, x_{(1,5)}}^+x_{(1,2)}x_{(3,5)}+p_{K_{3}, x_{(1,5)}}^-x_{(1,3)}x_{(2,5)}\\
=\, & p_{K_{32}, x_{(2,3)}}^-x_{(1,2)}x_{(3,5)}+p_{K_{32}, x_{(2,3)}}^+x_{(1,3)}x_{(2,5)}
\end{align*}
with $p_{K_{32}, x_{(2,3)}}^\pm=p_{K_{3}, x_{(1,5)}}^\mp$.

\item There is the following exchange relation between $\mathcal{X}(K_{11})$ and $\mathcal{X}(K_{31})$.
\begin{align*}
x_{(1,4)}x_{(3,5)}\,=\, & p_{K_{11}, x_{(1,4)}}^+x_{(1,3)}x_{(4,5)}+p_{K_{11}, x_{(1,4)}}^-x_{(1,5)}x_{(3,4)} \\
=\, & p_{K_{31}, x_{(3,5)}}^-x_{(1,3)}x_{(4,5)}+p_{K_{31}, x_{(3,5)}}^+x_{(1,5)}x_{(3,4)}
\end{align*}
with $p_{K_{11}, x_{(1,4)}}^\pm=p_{K_{31}, x_{(3,5)}}^\mp$.

\item There is the following exchange relation between $\mathcal{X}(K_{11})$ and $\mathcal{X}(K_{22})$.
\begin{align*}
x_{(2,4)}x_{(3,5)}\,=\, & p_{K_{11}, x_{(2,4)}}^+x_{(2,3)}x_{(4,5)}+p_{K_{11}, x_{(2,4)}}^-x_{(2,5)}x_{(3,4)} \\
=\, & p_{K_{22}, x_{(3,5)}}^-x_{(2,3)}x_{(4,5)}+p_{K_{22}, x_{(3,5)}}^+x_{(2,5)}x_{(3,4)}
\end{align*}
with $p_{K_{11}, x_{(2,4)}}^\pm=p_{K_{22}, x_{(3,5)}}^\mp$.

\item There is the following exchange relation between $\mathcal{X}(K_{12})$ and $\mathcal{X}(K_{21})$.
\begin{align*}
x_{(2,5)}x_{(3,4)}\,=\, & p_{K_{12}, x_{(2,5)}}^+x_{(2,3)}x_{(4,5)}+p_{K_{12}, x_{(2,5)}}^-x_{(2,4)}x_{(3,5)}\\
=\, & p_{K_{21}, x_{(3,4)}}^-x_{(2,3)}x_{(4,5)}+p_{K_{21}, x_{(3,4)}}^+x_{(2,4)}x_{(3,5)}
\end{align*}
with $p_{K_{12}, x_{(2,5)}}^\pm=p_{K_{21}, x_{(3,4)}}^\mp$.

\item There is the following exchange relation between $\mathcal{X}(K_{12})$ and $\mathcal{X}(K_{32})$.
\begin{align*}
x_{(1,5)}x_{(3,4)}\,=\, & p_{K_{12}, x_{(1,5)}}^+x_{(1,3)}x_{(4,5)}+p_{K_{12}, x_{(1,5)}}^-x_{(1,4)}x_{(3,5)}\\
=\, & p_{K_{32}, x_{(3,4)}}^-x_{(1,3)}x_{(4,5)}+p_{K_{32}, x_{(3,4)}}^+x_{(1,4)}x_{(3,5)}
\end{align*}
with $p_{K_{12}, x_{(1,5)}}^\pm=p_{K_{32}, x_{(3,4)}}^\mp$.

\item There is the following exchange relation between $\mathcal{X}(K_{21})$ and $\mathcal{X}(K_{31})$.
\begin{align*}
x_{(1,4)}x_{(2,5)}\,=\, & p_{K_{21}, x_{(1,4)}}^+x_{(1,2)}x_{(4,5)}+p_{K_{21}, x_{(1,4)}}^-x_{(1,5)}x_{(2,4)} \\
=\, & p_{K_{31}, x_{(2,5)}}^-x_{(1,2)}x_{(4,5)}+p_{K_{31}, x_{(2,5)}}^+x_{(1,5)}x_{(2,4)}
\end{align*}
with $p_{K_{21}, x_{(1,4)}}^\pm=p_{K_{31}, x_{(2,5)}}^\mp$.

\item There is the following exchange relation between $\mathcal{X}(K_{22})$ and $\mathcal{X}(K_{32})$.
\begin{align*}
x_{(1,5)}x_{(2,4)}\,=\, & p_{K_{22}, x_{(1,5)}}^+x_{(1,2)}x_{(4,5)}+p_{K_{22}, x_{(1,5)}}^-x_{(1,4)}x_{(2,5)}\\
=\, & p_{K_{32}, x_{(2,4)}}^-x_{(1,2)}x_{(4,5)}+p_{K_{32}, x_{(2,4)}}^+x_{(1,4)}x_{(2,5)}
\end{align*}
with $p_{K_{22}, x_{(1,5)}}^\pm=p_{K_{32}, x_{(2,4)}}^\mp$.
\end{enumerate}

\begin{lemma} We have the following relations:
\begin{align}
x_{(1,2)}x_{(4,5)}\,&=\,x_{(1,4)}x_{(2,5)}\,=\,x_{(1,5)}x_{(2,4)}\label{ebr-1}\\
x_{(1,3)}x_{(4,5)}\,&=\,x_{(1,4)}x_{(3,5)}\,=\,x_{(1,5)}x_{(3,4)}\label{ebr-2}\\
x_{(2,3)}x_{(4,5)}\,&=\,x_{(2,4)}x_{(3,5)}\,=\,x_{(2,5)}x_{(3,4)}\label{ebr-3}\\
x_{(1,2)}x_{(3,5)}\,&=\,x_{(1,3)}x_{(2,5)}\,=\,x_{(1,5)}x_{(2,3)}\label{ebr-4}\\
x_{(1,2)}x_{(3,4)}\,&=\,x_{(1,3)}x_{(2,4)}\,=\,x_{(1,4)}x_{(2,3)}\label{ebr-5}.
\end{align}
\end{lemma}

\begin{proof} We only prove that~\eqref{ebr-1} holds since the proofs of other cases~\eqref{ebr-2}--\eqref{ebr-5} are similar. Let
\[
p_{K, x_{(1,2)}}^+/p_{K, x_{(1,2)}}^-\,=\,u,\quad p_{K_{21}, x_{(1,4)}}^+/p_{K_{21}, x_{(1,4)}}^-\,=\,v\quad\text{and}\quad p_{K_{22}, x_{(1,5)}}^+/p_{K_{22}, x_{(1,5)}}^-\,=\,w.
\]
Since $p_{K, x_{(1,2)}}^+\oplus p_{K, x_{(1,2)}}^-=1$, one has that $p_{K, x_{(1,2)}}^+=\frac{u}{1\oplus u}$ and $p_{K, x_{(1,2)}}^-=\frac{1}{1\oplus u}$. Similarly, one has that $p_{K_{21}, x_{(1,4)}}^+=\frac{v}{1\oplus v}$ and $p_{K_{21}, x_{(1,4)}}^-=\frac{1}{1\oplus v}$, and $p_{K_{22}, x_{(1,5)}}^+=\frac{w}{1\oplus w}$ and $p_{K_{22}, x_{(1,5)}}^-=\frac{1}{1\oplus w}$. Thus we rewrite (1), (14) and (15) as follows:
\begin{align}
x_{(1,2)}x_{(4,5)}\,&=\,\frac{u}{1\oplus u}x_{(1,4)}x_{(2,5)}+\frac{1}{1\oplus u}x_{(1,5)}x_{(2,4)}\label{re-1}\\
x_{(1,4)}x_{(2,5)}\,&=\,\frac{v}{1\oplus v}x_{(1,2)}x_{(4,5)}+\frac{1}{1\oplus v}x_{(1,5)}x_{(2,4)}\label{re-2}\\
x_{(1,5)}x_{(2,4)}\,&=\,\frac{w}{1\oplus w}x_{(1,2)}x_{(4,5)}+\frac{1}{1\oplus w}x_{(1,4)}x_{(2,5)}.\label{re-3}
\end{align}
From~\eqref{re-1} and~\eqref{re-2} we have that
\begin{align*}
x_{(1,2)}x_{(4,5)}\,=&\, \frac{u}{1\oplus u}\cdot\frac{v}{1\oplus v}x_{(1,2)}x_{(4,5)}+\left(\frac{u}{1\oplus u}\cdot\frac{1}{1\oplus v}\oplus \frac{1}{1\oplus u}\right)x_{(1,5)}x_{(2,4)}\\
=&\, \frac{uv}{(1\oplus u)(1\oplus v)}x_{(1,2)}x_{(4,5)}+\frac{1\oplus u\oplus v}{(1\oplus u)(1\oplus v)}x_{(1,5)}x_{(2,4)}
\end{align*}
so $(1\oplus u\oplus v)x_{(1,2)}x_{(4,5)}=(1\oplus u\oplus v)x_{(1,5)}x_{(2,4)}$. Thus, $x_{(1,2)}x_{(4,5)}=x_{(1,5)}x_{(2,4)}$. In a similar way as above, one can obtain from~\eqref{re-1} and~\eqref{re-3} that $x_{(1,2)}x_{(4,5)}=x_{(1,4)}x_{(2,5)}$. This shows that~\eqref{ebr-1} holds.
\end{proof}

\begin{remark}
The relation~\eqref{ebr-2} is determined by (2), (10) and (13);~\eqref{ebr-3} is determined by (3), (11) and (12);~\eqref{ebr-4} is determined by (4), (6) and (9);~\eqref{ebr-5} is determined by (5), (7) and (8). In addition, it is easy to see that~\eqref{ebr-4} and~\eqref{ebr-5} can be induced from~\eqref{ebr-1}--\eqref{ebr-3}.
\end{remark}

\begin{proposition}
As a subalgebra of $\mathbb{F}$, $\mathcal{A}_{[K]}$ is the quotient algebra of the polynomial algebra
\[
\mathbb{ZP}[x\mid x\in \mathcal{X}_{[K]}]\,=\, \mathbb{ZP}[x_{(1,2)}, x_{(1,3)}, x_{(1,4)}, x_{(1,5)}, x_{(2,3)}, x_{(2,4)}, x_{(2,5)}, x_{(3,4)}, x_{(3,5)}, x_{(4,5)}]
\]
as follows:
\[
\mathcal{A}_{[K]}\,=\,\mathbb{ZP}[x\mid x\in \mathcal{X}_{[K]}]/\mathcal{I},
\]
where $\mathcal{I}$ is the ideal determined by those relations~\eqref{ebr-1}--\eqref{ebr-3}.

As a subalgebra in $\mathbb{F}\mathcal{X}(K)$ with $(\mathcal{X}(K), {\bf p}_K, B(K))$ as an initial bistellar seed, $\mathcal{A}_K\cong \mathcal{A}_{[K]}$ is generated by $x_{(1,2)}, x_{(1,3)}, x_{(1,4)}, x_{(1,5)}, x_{(2,3)}, x_{(2,4)}, x_{(2,5)}, x_{(3,4)}, x_{(3,5)}$ and
\[
\frac{x_{(1,4)}x_{(2,5)}}{x_{(1,2)}}\,=\, \frac{x_{(1,5)}x_{(2,4)}}{x_{(1,2)}}\,=\,\frac{x_{(1,4)}x_{(3,5)}}{x_{(1,3)}}\,=\,\frac{x_{(1,5)}x_{(3,4)}}{x_{(1,3)}}\,=\, \frac{x_{(2,4)}x_{(3,5)}}{x_{(2,3)}}\,=\,\frac{x_{(2,5)}x_{(3,4)}}{x_{(2,3)}}.
\]
\end{proposition}


\section{Algebras as PL Invariants}

In this section, our purpose is to construct PL invariants by using bistellar cluster algebras.  We shall work in the setting of PL manifolds. Let $\mathbb{PL}_n$ be the set consisting of all $n$-dimensional PL manifolds in $\mathbb{TM}_n$. We shall mainly be concerned with the case in which $n$ is even, especially for $n=4$.

Write $n=2h$  and take a PL manifold $K\in \mathbb{PL}_n$. Let $\mathfrak{S}_K$ be the set consisting of all PL manifolds that are PL homeomorphic to $K$. Then we can write
\[
\mathfrak{S}_K\,=\,\bigcup_{m\geq m_0}\mathfrak{S}_K^m,
\]
where $m_0$ is the minimal number of vertex numbers of all PL manifolds in $\mathfrak{S}_K$, and $\mathfrak{S}_K^m$ consists of those PL manifolds with $m$~vertices in $\mathfrak{S}_K$.

Consider the equivalence relation $\sim$ on $\mathfrak{S}_K$ as follows: for $K_1, K_2\in \mathfrak{S}_K$, $K_1\sim K_2$ if and only if one of $K_1$ and $K_2$ can be transformed into the other by only doing a finite sequence of bistellar $h$-moves (including the empty sequence). By $\widetilde{\mathfrak{S}}_K$ we denote the quotient set $\mathfrak{S}_K/\sim$. Since the bistellar $h$-moves do not change the number of vertices,
the equivalence relation $\sim$ is graded on $\mathfrak{S}_K=\bigcup_{m\geq m_0}\mathfrak{S}_K^m$, so
\[
\widetilde{\mathfrak{S}}_K\,=\,\bigcup_{m\geq m_0}\widetilde{\mathfrak{S}}_K^m.
\]
For a PL manifold $L$ in $\mathfrak{S}_K$, the equivalence class of $L$ in $\widetilde{\mathfrak{S}}_K$ is denoted by $[L]$.

We define a binary relation $\preceq$ on $\widetilde{\mathfrak{S}}_K$ as follows: for two elements $[K_1], [K_2]\in \widetilde{\mathfrak{S}}_K$ we say
\begin{align*}
[K_1]\preceq [K_2]\quad\Longleftrightarrow\quad &\text{$K_2$ is obtained from  $K_1$ by  doing a finite sequence of  bistellar} \\
&\ \text{ moves ${\bm}_{\alpha_1},\dots, {\bm}_{\alpha_l}$, where $\dim \alpha_i\geq  h$}.
\end{align*}
Note that for $[K']\in \widetilde{\mathfrak{S}}_K$, if $|[K']|=1$, then we use the convention that $[K']\preceq [K']$.

\begin{lemma}\label{dir-set}
When $h=1, 2$, $(\widetilde{\mathfrak{S}}_K, \preceq)$ forms a directed set.
\end{lemma}

\begin{proof}
When $h=1$, let $[L]\in \widetilde{\mathfrak{S}}_K$. Clearly, $[L]\preceq [L]$. It is well-known that for any $m\geq m_0$ and  any two $L_1, L_2\in \mathfrak{S}_K^m$, that $L_1$ can be obtained from $L_2$ by performing a finite sequence of bistellar 1-moves (this can also be seen by simplifying the case for $h=2$ below). 
So $\widetilde{\mathfrak{S}}^m_K$ contains exactly one element. It then follows easily from this that $(\widetilde{\mathfrak{S}}_K, \preceq)$ forms a directed set.

When $h=2$, let $[L]\in \widetilde{\mathfrak{S}}_K$. First, clearly one has that $[L]\preceq [L]$. Let $[L_1], [L_2], [L_3]\in \widetilde{\mathfrak{S}}_K$ be such that $[L_1]\preceq [L_2]$ and $[L_2]\preceq [L_3]$. Then  $L_2$ is obtained from $L_1$ by performing finitely many bistellar $0$-moves, $1$-moves or $2$-moves, and $L_3$ is obtained from $L_2$ similarly. So $L_3$ is obtained from $L_1$ by performing finitely many bistellar $0$-moves, $1$-moves or $2$-moves. Thus, $[L_1]\preceq [L_3]$. It remains to show that for any two $[L_1], [L_2]\in \widetilde{\mathfrak{S}}_K$, there is an $[L]\in \widetilde{\mathfrak{S}}_K$ such that $[L_1]\preceq [L]$ and $[L_2]\preceq [L]$.

Clearly if either $L_1$ is obtained from  $L_2$ by performing finitely many bistellar $0$-moves, $1$-moves or $2$-moves or  $L_2$ is obtained from  $L_1$ in a similar fashion, then we can choose $[L]$ as $[L_1]$ or $[L_2]$ such that $[L_1]\preceq [L]$ and $[L_2]\preceq [L]$. Generally, by Pachner's Theorem~\cite{pa2}, there always exists a finite sequence of bistellar pairs $(\alpha_1, \beta_1),\dots, (\alpha_l, \beta_l)$ such that
  \begin{equation}\label{move sequence}
  {\bm}_{\alpha_l}\cdots{\bm}_{\alpha_1}L_1\,=\,L_2,
  \end{equation}
  where $0\leq \dim \alpha_i\leq 4$.  Such a sequence $(\alpha_1, \beta_1),\dots, (\alpha_l, \beta_l)$ is not unique in general. To complete the proof, it suffices to show that there exists such a sequence $(\alpha_1, \beta_1), \dots, (\alpha_l, \beta_l)$,  satisfying~\eqref{move sequence} and the property that  there is some $1\leq u\leq l$ such that for $1\leq i\leq u$, $\dim \alpha_i\geq 2$ and for $u<i\leq l$, $\dim \alpha_i<2$, so that we can obtain $L={\bm}_{\alpha_u}\cdots{\bm}_{\alpha_1}L_1={\bm}_{\beta_{u+1}}\cdots{\bm}_{\beta_l}L_2$ as desired.

For this we let $L\in [K]$, and then it suffices to show that for any two bistellar pairs $(\alpha, \beta)$ with $\dim \alpha<2$ and $(\alpha', \beta')$ with $\dim \alpha'\geq 2$, if  we can perform the operation ${\bm}_{\alpha'}{\bm}_\alpha$ on $L$, then  ${\bm}_{\alpha'}{\bm}_\alpha L$ can be written as ${\bm}_{\alpha_l}\cdots{\bm}_{\alpha_u}\cdots{\bm}_{\alpha_1}L$ such that for $1\leq i\leq u$, $\dim \alpha_i\geq 2$ and for $u<i\leq l$, $\dim \alpha_i<2$. We divide this into two cases: (I) $\dim \alpha=0$; (II) $\dim \alpha=1$.

\begin{enumerate}
\item[{\bf (I)}] Without loss of generality, we assume that $\alpha=(1)$ and $\beta=(2,3,4,5,6)$, so
  $\alpha\ast\partial \beta\subset L$. Then $\beta=\partial \alpha\ast\beta\subset {\bf bm}_\alpha L$.

  If $\dim \alpha'=4$ and $\alpha'=\beta$ (i.e., $|\alpha'\cap\beta|=5$), then ${\bf bm}_{\alpha'}{\bf bm}_\alpha L=L$.
  In this case, first we do a bistellar 0-move at $(1,2,3,4,5)$ on $L$, so that $(1,2,3,4,5)$ becomes a missing face of ${\bf bm}_{(1,2,3,4,5)}L$, and let $(6')\in {\bf bm}_{(1,2,3,4,5)}L$ such that $\Link_L(1,2,3,4,5)=\partial (6')$. Then
  we do the inverse operation of  ${\bf bm}_{(1,2,3,4,5)}$, so that
  \[
  {\bf bm}_{\alpha'}{\bf bm}_\alpha L\,=\,{\bf bm}_{(2,3,4,5,6)}{\bf bm}_{(1)} L\,=\,L\,=\,{\bf bm}_{(6')}{\bf bm}_{(1,2,3,4,5)}L.
  \]

 If $\dim \alpha'=4$ and  $|\alpha'\cap\beta|\leq 4$, without loss of generality,  we may assume that $\alpha'=(2,3,4,5,6')$ or $(2,3,4,5',6')$ or $(2,3, 4',5',6')$ or $(2, 3', 4', 5', 6')$ or $(2',3',4',5',6')$. Clearly, $\alpha'$ is a simplex of $L$ so
 $\Link_L\alpha'$ is the boundary of a ghost vertex $\beta'=v$ (i.e., $\beta'\not\in L$). Thus
 $\alpha\in {\bf bm}_{\alpha'}L$ and $\alpha'\in {\bf bm}_{\alpha}L$.
It is easy to check that
 \[
 {\bf bm}_{\alpha'}{\bf bm}_\alpha L\,=\,{\bf bm}_{\alpha}{\bf bm}_{\alpha'} L
 \]
as required.

    If $\dim \alpha'=3$ and  $|\alpha'\cap\beta|=4$, without loss of generality,  assume that $\alpha'=(2,3,4,5)$, then we can write
 $\beta'=(6,6')\not\in  {\bf bm}_\alpha L$. We easily check that ${\bf bm}_{\alpha'}{\bf bm}_\alpha L$ must contain the following four 4-simplices $(2,3,4,6,6')$, $(2,3,5,6,6')$, $(2,4,5,6,6')$ and $(3,4,5,6,6')$. Now let us perform the following bistellar moves on $L$.
 \begin{enumerate}
\item
Firstly we  see that $(1,6')\not\in L$ and $\Link_L(2,3,4,5)=\partial (1,6')$, so
 we can do a bistellar 1-move at $(2,3,4,5)$ on $L$.
 \item
 Secondly we have that $(6,6')\not\in {\bf bm}_{(2,3,4,5)} L$
 and $\Link_{{\bf bm}_{(2,3,4,5)}L} (1,2,3,4)=\partial (6,6')$, so we can do a bistellar 1-move at $(1,2,3,4)$ on
 ${\bf bm}_{(2,3,4,5)} L$, and in particular, we have $(2,3,4,6,6')\in {\bf bm}_{(1,2,3,4)}{\bf bm}_{(2,3,4,5)} L$.
 \item
  Thirdly we have that $(5,6,6')\not\in {\bf bm}_{(1,2,3,4)}{\bf bm}_{(2,3,4,5)} L$ and
 \[
 \Link_{{\bf bm}_{(1,2,3,4)}{\bf bm}_{(2,3,4,5)} L} (1,3,4)\,=\,\partial (5,6,6').
 \]
  Then we  perform a bistellar 2-move at $(1, 3,4)$ on ${\bf bm}_{(1,2,3,4)}{\bf bm}_{(2,3,4,5)} L$, so that
  $(3,4,5,6,6')\in {\bf bm}_{(1,3,4)}{\bf bm}_{(1,2,3,4)}{\bf bm}_{(2,3,4,5)} L$.
 \item
  Next it is easy to see that $(2, 5, 6,6')\not\in {\bf bm}_{(1,3,4)}{\bf bm}_{(1,2,3,4)}{\bf bm}_{(2,3,4,5)} L$ and
 \[
 \Link_{{\bf bm}_{(1,3,4)}{\bf bm}_{(1,2,3,4)}{\bf bm}_{(2,3,4,5)} L}(1,4)\,=\,\partial (2,5,6,6').
 \]
 After we perform a bistellar 3-move at $(1,4)$ on ${\bf bm}_{(1,3,4)}{\bf bm}_{(1,2,3,4)}{\bf bm}_{(2,3,4,5)} L$,
 we have
 \[
 (2,4,5,6,6')\,\in\, {\bf bm}_{(1,4)} {\bf bm}_{(1,3,4)}{\bf bm}_{(1,2,3,4)}{\bf bm}_{(2,3,4,5)} L.
 \]
 \item
 Finally we can check that
 \[
 \Link_{{\bf bm}_{(1,4)}{\bf bm}_{(1,3,4)}{\bf bm}_{(1,2,3,4)}{\bf bm}_{(2,3,4,5)} L}(1)\,=\,\partial (2,3,5,6,6')
 \]
 but $(2,3,5,6,6')\not\in {\bf bm}_{(1,4)}{\bf bm}_{(1,3,4)}{\bf bm}_{(1,2,3,4)}{\bf bm}_{(2,3,4,5)} L$. After performing a bistellar 4-move at $(1)$ on ${\bf bm}_{(1,4)}{\bf bm}_{(1,3,4)}{\bf bm}_{(1,2,3,4)}{\bf bm}_{(2,3,4,5)} L$, we obtain that
 \[
 (2,3,5,6,6')\in {\bf bm}_{(1)}{\bf bm}_{(1,4)}{\bf bm}_{(1,3,4)}{\bf bm}_{(1,2,3,4)}{\bf bm}_{(2,3,4,5)} L.
 \]
 \end{enumerate}
  Moreover, an easy argument shows
 \[
 {\bf bm}_{\alpha'}{\bf bm}_\alpha L\,=\,{\bf bm}_{(1)}{\bf bm}_{(1,4)}{\bf bm}_{(1,3,4)}{\bf bm}_{(1,2,3,4)}{\bf bm}_{(2,3,4,5)} L
 \]
 as required.

     If $\dim \alpha'=3$ and  $|\alpha'\cap\beta|\leq 3$, without loss of generality, we may assume that $\alpha'=(2,3,4, 5')$ or $(2,3, 4',5')$
 or $(2,3',4',5')$ or $(2',3',4',5')$. Similarly to the argument of the case $\dim \alpha'=4$ and  $|\alpha'\cap\beta|\leq 4$,
 we have the following equation
  \[
  {\bf bm}_{\alpha'}{\bf bm}_\alpha L\,=\,{\bf bm}_{\alpha}{\bf bm}_{\alpha'} L
  \]
as required.

As in the above case, if $\dim \alpha'=2$ and $|\alpha'\cap \beta|\leq 2$, then we easily see that
\[
{\bf bm}_{\alpha'}{\bf bm}_\alpha L\,=\,{\bf bm}_{\alpha}{\bf bm}_{\alpha'} L.
\]
If $\dim \alpha'=2$ and $|\alpha'\cap \beta|=3$, we may assume that $\alpha'=(2,3,4)$. We see that $\Link_L
\alpha'$ is a square formed by four edges $(1,5)$, $(1,6)$, $(5, 6')$ and $(6,6')$. This implies that $\beta'=(5,6,6')$, which is a 2-simplex of ${\bf bm}_{\alpha'}{\bf bm}_\alpha L$ but not a 2-simplex of $L$ and ${\bf bm}_\alpha L$. A direct check gives that  ${\bf bm}_{\alpha'}{\bf bm}_\alpha L$ must contain the following three 4-simplices $(2,3,5,6,6')$, $(2,4,5,6,6')$ and  $(3,4,5,6,6')$. Now for our purpose, we perform the following bistellar moves on $L$:
\begin{enumerate}
\item
Firstly we perform a bistellar 1-move at $(2,3,4,5)$ on $L$ since  $(1,6')\not\in L$ and $\Link_L(2,3,4,5)=\partial (1,6')$.
 \item
 Secondly we can perform a bistellar 2-move at $(2,3,4)$ on
 ${\bf bm}_{(2,3,4,5)} L$. This is because   $(1,6,6')\not\in {\bf bm}_{(2,3,4,5)} L$ and $\Link_{{\bf bm}_{(2,3,4,5)}L}(2,3,4)=\partial (1,6,6')$.
 \item
 Thirdly we see that $(5,6,6')\not\in {\bf bm}_{(2,3,4)}{\bf bm}_{(2,3,4,5)} L$ and
 \[
 \Link_{{\bf bm}_{(2,3,4)}{\bf bm}_{(2,3,4,5)} L} (1,3,4)\,=\,\partial (5,6,6').
 \]
  Thus we can perform a bistellar 2-move at $(1, 3,4)$ on ${\bf bm}_{(2,3,4)}{\bf bm}_{(2,3,4,5)} L$, so that
  $(3,4,5,6,6')\in {\bf bm}_{(1,3,4)}{\bf bm}_{(2,3,4)}{\bf bm}_{(2,3,4,5)} L$.
 \item
  Next we can check that $(2, 5, 6,6')\not\in {\bf bm}_{(1,3,4)}{\bf bm}_{(2,3,4)}{\bf bm}_{(2,3,4,5)} L$ and
 \[
 \Link_{{\bf bm}_{(1,3,4)}{\bf bm}_{(2,3,4)}{\bf bm}_{(2,3,4,5)} L}(1,4)\,=\,\partial (2,5,6,6').
 \]
 After we perform a bistellar 3-move at $(1,4)$ on ${\bf bm}_{(1,3,4)}{\bf bm}_{(2,3,4)}{\bf bm}_{(2,3,4,5)} L$, we have
 \[
 (2,4,5,6,6')\,\in\, {\bf bm}_{(1,4)} {\bf bm}_{(1,3,4)}{\bf bm}_{(2,3,4)}{\bf bm}_{(2,3,4,5)} L.
 \]
 \item
 Finally we can see that $\Link_{{\bf bm}_{(1,4)}{\bf bm}_{(1,3,4)}{\bf bm}_{(2,3,4)}{\bf bm}_{(2,3,4,5)} L}(1)=\partial (2,3,5,6,6')$
 but
 \[
 (2,3,5,6,6')\not\in {\bf bm}_{(1,4)}{\bf bm}_{(1,3,4)}{\bf bm}_{(2,3,4)}{\bf bm}_{(2,3,4,5)} L.
 \]
 Thus, after performing a bistellar 4-move at $(1)$ on
 \[
 {\bf bm}_{(1,4)}{\bf bm}_{(1,3,4)}{\bf bm}_{(2,3,4)}{\bf bm}_{(2,3,4,5)} L,
 \]
 we conclude that
 \[
 (2,3,5,6,6')\,\in\, {\bf bm}_{(1)}{\bf bm}_{(1,4)}{\bf bm}_{(1,3,4)}{\bf bm}_{(2,3,4)}{\bf bm}_{(2,3,4,5)} L.
 \]
 \end{enumerate}
  Together with the above arguments, we may obtain the required equation
 \[
 {\bf bm}_{\alpha'}{\bf bm}_\alpha L\,=\,{\bf bm}_{(1)}{\bf bm}_{(1,4)}{\bf bm}_{(1,3,4)}{\bf bm}_{(2,3,4)}{\bf bm}_{(2,3,4,5)}. L
 \]

\item[\bf (II)] Without loss of generality, we assume that $\alpha=(1,2)$ and $\beta=(3,4,5,6)$.

If $\dim \alpha'=4$ and $|\alpha'\cap(\alpha\cup \beta)|=5$, then $\alpha'\in {\bf H}_\beta=\{(1,3,4,5,6), (2,3,4,5,6)\}$. Without loss of generality, we may assume that $\alpha'=(1,2,3,4,5)$. Let $\beta'=(7)$, which is a ghost vertex of $L$ and ${\bf bm}_\alpha L$. Then ${\bf bm}_{\alpha'}{\bf bm}_\alpha L$ must contain the following six 4-simplices $(2,3,4,5,6)$, $(3,4,5,6,7)$, $(1,3,4, 5,7)$, $(1,3,4,6,7)$, $(1,3,5,6,7)$ and $(1,4,5,6,7)$.
   In a similar way to case~(I),  we perform the following bistellar moves on $L$:
\begin{enumerate}
\item
Firstly we perform a bistellar $0$-move at $(1,2,3,4,5)$ on $L$ since  $\Link_L(1,2,3,4,5)= \partial (7)$. Then we easily see that $(1,3,4,5,7)\in {\bf bm}_{(1,2,3,4,5)}L$.
 \item
 Secondly we can perform a bistellar 1-move at $(1,2,3,4)$ on
 ${\bf bm}_{(1,2,3,4,5)} L$ since  $(6,7)\not\in {\bf bm}_{(1,2,3,4,5)} L$ and $\Link_{{\bf bm}_{(1,2, 3,4,5)}L}(1,2,3,4)=\partial (6,7)$, so that we can obtain that
 \[
 (1, 3,4,6,7)\,\in\, {\bf bm}_{(1,2,3,4)} {\bf bm}_{(1,2,3,4,5)} L.
 \]
 \item
  Thirdly we see that $(5,6,7)\not\in {\bf bm}_{(1,2,3,4)}{\bf bm}_{(1,2,3,4,5)} L$ and
 \[
 \Link_{{\bf bm}_{(1,2,3,4)}{\bf bm}_{(1,2,3,4,5)} L} (1,2,3)\,=\,\partial (5,6,7).
 \]
  Thus we can perform a bistellar 2-move at $(1, 2,3)$ on ${\bf bm}_{(1,2,3,4)}{\bf bm}_{(1,2,3,4,5)} L$, so that
  $(1,3,5,6,7)\in {\bf bm}_{(1,2,3)}{\bf bm}_{(1,2,3,4)}{\bf bm}_{(1,2,3,4,5)} L$.
 \item
  Next we can check that $(4,5,6,7)\not\in {\bf bm}_{(1,2,3)}{\bf bm}_{(1,2,3,4)}{\bf bm}_{(1,2,3,4,5)} L$ and
 \[
 \Link_{{\bf bm}_{(1,2,3)}{\bf bm}_{(1,2,3,4)}{\bf bm}_{(1,2,3,4,5)} L}(1,2)\,=\,\partial (4,5,6,7),
 \]
  we can perform a bistellar 3-move at $(1,2)$ on ${\bf bm}_{(1,2,3)}{\bf bm}_{(1,2,3,4)}{\bf bm}_{(1,2,3,4,5)} L$, and then
 $(1,4,5,6,7)\in {\bf bm}_{(1,2)} {\bf bm}_{(1,2,3)}{\bf bm}_{(1,2,3,4)}{\bf bm}_{(1,2,3,4,5)} L$.
 \item
 Finally we can see that $\Link_{{\bf bm}_{(1,2)}{\bf bm}_{(1,2,3)}{\bf bm}_{(1,2,3,4)}{\bf bm}_{(1,2,3,4,5)} L}(2,7)=\partial (3,4,5,6)$
 but $(3,4,5,6)\not\in {\bf bm}_{(1,2)}{\bf bm}_{(1,2,3)}{\bf bm}_{(1,2,3,4)}{\bf bm}_{(1,2,3,4,5)} L$. Thus, after performing a bistellar 3-move at $(2,7)$ on ${\bf bm}_{(1,2)}{\bf bm}_{(1,2,3)}{\bf bm}_{(1,2,3,4)}{\bf bm}_{(1,2,3,4,5)} L$, we conclude that
 \[
 (2,3,4,5,6), (3,4,5,6,7)\,\in\, {\bf bm}_{(2,7)}{\bf bm}_{(1,2)}{\bf bm}_{(1,2,3)}{\bf bm}_{(1,2,3,4)}{\bf bm}_{(1,2,3,4,5)} L.
 \]
 \end{enumerate}
  The above arguments give
  \[
  {\bf bm}_{\alpha'}{\bf bm}_\alpha L\,=\,{\bf bm}_{(2,7)}{\bf bm}_{(1,2)}{\bf bm}_{(1,2,3)}{\bf bm}_{(1,2,3,4)}{\bf bm}_{(1,2,3,4,5)} L.
  \]

 If $\dim \alpha'=4$ and $|\alpha'\cap(\alpha\cup \beta)|\leq 4$, then an easy observation shows that
 \[
 {\bf bm}_{\alpha'}{\bf bm}_\alpha L\,=\, {\bf bm}_{\alpha}{\bf bm}_{\alpha'} L.
 \]
Actually, we can also prove easily that  if $\dim \alpha'=i$ and $|\alpha'\cap(\alpha\cup \beta)|\leq i$ for $i=2,3$, then
 \[
 {\bf bm}_{\alpha'}{\bf bm}_\alpha L\,=\, {\bf bm}_{\alpha}{\bf bm}_{\alpha'} L.
 \]
Thus it remains to show the case in which   $\dim \alpha'=i$ and $|\alpha'\cap(\alpha\cup \beta)|=i+1$ for $i=2,3$.

When $i=2$, without a loss, we may assume that $\alpha'=(1,3,4)$. Then $\beta'=(5,6,6')$, and further $ {\bf bm}_{\alpha'}{\bf bm}_\alpha L$ must contain the following three $4$-simplices
$(1,3, 5,6,6')$, $(1,4,5,6,6')$ and $(3,4,5,6,6')$.
Next we perform the following bistellar moves on $L$:
\begin{enumerate}
\item
Firstly we perform a bistellar $1$-move at $(1,3,4,5)$ on $L$ since  $\Link_L(1,3,4,5)= \partial (2,6')$.
 \item
 Secondly   we  perform a bistellar 2-move at $(1,3,4)$ on
 ${\bf bm}_{(1,3,4,5)} L$ since  $(2,6,6')\not\in {\bf bm}_{(1,3,4,5)} L$ and $\Link_{{\bf bm}_{(1, 3,4,5)}L}(1,3,4)=\partial (2,6,6')$.
 \item
  Thirdly we see that $(5,6,6')\not\in {\bf bm}_{(1,3,4)}{\bf bm}_{(1,3,4,5)} L$ and
 \[
 \Link_{{\bf bm}_{(1,3,4)}{\bf bm}_{(1,3,4,5)} L} (1,2,3)\,=\,\partial (5,6,6').
 \]
  Thus we can perform a bistellar 2-move at $(1, 2,3)$ on ${\bf bm}_{(1,3,4)}{\bf bm}_{(1,3,4,5)} L$, so that
  $(1,3,5,6,6')\in {\bf bm}_{(1,2,3)}{\bf bm}_{(1,3,4)}{\bf bm}_{(1,3,4,5)} L$.
 \item
  Next we can check that $(4,5,6,6')\not\in {\bf bm}_{(1,2,3)}{\bf bm}_{(1,3,4)}{\bf bm}_{(1,3,4,5)} L$ and
 \[
 \Link_{{\bf bm}_{(1,2,3)}{\bf bm}_{(1,3,4)}{\bf bm}_{(1,3,4,5)} L}(1,2)\,=\,\partial (4,5,6,6'),
 \]
  we can perform a bistellar 3-move at $(1,2)$ on ${\bf bm}_{(1,2,3)}{\bf bm}_{(1,3,4)}{\bf bm}_{(1,3,4,5)} L$, and then
 $$(1,4,5,6,6')\in {\bf bm}_{(1,2)} {\bf bm}_{(1,2,3)}{\bf bm}_{(1,3,4)}{\bf bm}_{(1,3,4,5)} L.$$
 \item
 Finally we can see that $\Link_{{\bf bm}_{(1,2)}{\bf bm}_{(1,2,3)}{\bf bm}_{(1,3,4)}{\bf bm}_{(1,3,4,5)} L}(2,6')=\partial (3,4,5,6)$
 but $(3,4,5,6)\not\in {\bf bm}_{(1,2)}{\bf bm}_{(1,2,3)}{\bf bm}_{(1,3,4)}{\bf bm}_{(1,3,4,5)} L$. Thus, after performing a bistellar 3-move at $(2,6')$ on ${\bf bm}_{(1,2)}{\bf bm}_{(1,2,3)}{\bf bm}_{(1,3,4)}{\bf bm}_{(1,3,4,5)} L$, we conclude that
 \[
 (3,4,5,6,6')\,\in\, {\bf bm}_{(2,6')}{\bf bm}_{(1,2)}{\bf bm}_{(1,2,3)}{\bf bm}_{(1,3,4)}{\bf bm}_{(1,3,4,5)} L.
 \]
 \end{enumerate}
  The above arguments give the following equality
  \[
   {\bf bm}_{\alpha'}{\bf bm}_\alpha L\,=\,{\bf bm}_{(2,6')}{\bf bm}_{(1,2)}{\bf bm}_{(1,2,3)}{\bf bm}_{(1,3,4)}{\bf bm}_{(1,3,4,5)} L.
   \]

When $i=3$, without loss of generality, we assume that $\alpha'=(1,3,4,5)$. Then $\beta'=(6,6')$ and ${\bf bm}_{\alpha'}{\bf bm}_\alpha L$ must contain the following three $4$-simplices
$(2,3,4,5,6)$, $(1,3,4,6,6')$, $(1,3, 5,6,6')$, $(1,4,5,6,6')$ and $(3,4,5,6,6')$.
Now we perform the following bistellar moves on $L$:
\begin{enumerate}
\item
Firstly we perform a bistellar $1$-move at $(1,3,4,5)$ on $L$ since  $\Link_L(1,3,4,5)= \partial (2,6')$.
 \item
 Secondly   we  perform a bistellar 1-move at $(1,2,3,5)$ on
 ${\bf bm}_{(1,3,4,5)} L$ since  $(6,6')\not\in {\bf bm}_{(1,3,4,5)} L$ and $\Link_{{\bf bm}_{(1, 3,4,5)}L}(1,2,3,5)=\partial (6,6')$, so that we can obtain that
 $$(1,3,5,6,6')\in {\bf bm}_{(1,2,3,5)}{\bf bm}_{(1,3,4,5)} L.$$
 \item
  Thirdly we see that $(4,6,6')\not\in {\bf bm}_{(1,2,3,5)}{\bf bm}_{(1,3,4,5)} L$ and
 \[
 \Link_{{\bf bm}_{(1,2,3,5)}{\bf bm}_{(1,3,4,5)} L} (1,2,3)\,=\,\partial (4,6,6').
 \]
  Thus we can perform a bistellar 2-move at $(1, 2,3)$ on ${\bf bm}_{(1,2,3,5)}{\bf bm}_{(1,3,4,5)} L$, so that
  $(1,3,4,6,6')\in {\bf bm}_{(1,2,3)}{\bf bm}_{(1,2,3,5)}{\bf bm}_{(1,3,4,5)} L$.
 \item
  Next we can check that $(4,5,6,6')\not\in {\bf bm}_{(1,2,3)}{\bf bm}_{(1,2,3,5)}{\bf bm}_{(1,3,4,5)} L$ and
 \[
 \Link_{{\bf bm}_{(1,2,3)}{\bf bm}_{(1,2,3,5)}{\bf bm}_{(1,3,4,5)} L}(1,2)\,=\,\partial (4,5,6,6'),
 \]
  we can perform a bistellar 3-move at $(1,2)$ on ${\bf bm}_{(1,2,3)}{\bf bm}_{(1,2,3,5)}{\bf bm}_{(1,3,4,5)} L$, and then
 $$(1,4,5,6,6')\in {\bf bm}_{(1,2)} {\bf bm}_{(1,2,3)}{\bf bm}_{(1,2,3,5)}{\bf bm}_{(1,3,4,5)} L.$$
 \item
 Finally we see  that $\Link_{{\bf bm}_{(1,2)}{\bf bm}_{(1,2,3)}{\bf bm}_{(1,2,3,5)}{\bf bm}_{(1,3,4,5)} L}(2,6')=\partial (3,4,5,6)$
 but $(3,4,5,6)\not\in {\bf bm}_{(1,2)}{\bf bm}_{(1,2,3)}{\bf bm}_{(1,2,3,5)}{\bf bm}_{(1,3,4,5)} L$. Thus, after performing a bistellar 3-move at $(2,6')$ on ${\bf bm}_{(1,2)}{\bf bm}_{(1,2,3)}{\bf bm}_{(1,2,3,5)}{\bf bm}_{(1,3,4,5)} L$, we conclude that
 \[
 (2,3,4,5,6), (3,4,5,6,6')\,\in\, {\bf bm}_{(2,6')}{\bf bm}_{(1,2)}{\bf bm}_{(1,2,3)}{\bf bm}_{(1,2,3,5)}{\bf bm}_{(1,3,4,5)} L.
 \]
 \end{enumerate}
 Moreover,  the above arguments give the following equality
  \[
   {\bf bm}_{\alpha'}{\bf bm}_\alpha L\,=\,{\bf bm}_{(2,6')}{\bf bm}_{(1,2)}{\bf bm}_{(1,2,3)}{\bf bm}_{(1,2,3,5)}{\bf bm}_{(1,3,4,5)} L.
   \]
\end{enumerate}
\end{proof}

\begin{proposition}
For $h=1, 2$, $\{\mathcal{A}_{[L]}\}_{[L]\in \widetilde{\mathfrak{S}}_K}$ forms a direct system.
\end{proposition}

\begin{proof}
When $h=1$, the proof is not difficult. We would like to leave it as an exercise to the reader.

When $h=2$, by Lemma~\ref{dir-set} it suffices to show that for $[L]\preceq [L']$ in $\widetilde{\mathfrak{S}}_K$, there is a canonical embedding homomorphism $e_{[L][L']}\colon \mathcal{A}_{[L]}\hookrightarrow\mathcal{A}_{[L']}$. This is equivalent to showing that when $L'={\bm}_\alpha L$ with $\dim\alpha =3,4$, such a homomorphism $e_{[L][L']}$ exists. We will only prove the case of $\dim \alpha=3$ since the proof of the case $\dim\alpha=4$ is similar.

Let $\dim\alpha=3$. Without loss of generality, assume that $\alpha=(1,2,3,4)$ and $\beta=(5,6)$. Then
\[
\mathcal{X}(L')\,=\,\left(\mathcal{X}(L\right)\setminus\{x_{(1,2,3,4)}\})\cup \{x_{(1,2,5,6)}, x_{(1,3,5,6)}, x_{(1,4,5,6)}, x_{(2,3,5,6)}, x_{(2,4,5,6)}, x_{(3,4,5,6)}\}.
\]
Define $\varphi_\alpha\colon \mathcal{X}(L)\to\mathcal{X}(L')$ by
\[
\varphi_\alpha(x)\,=\,
\begin{cases}
x, & \text{if $x\in \mathcal{X}(L)\setminus\{x_{(1,2,3,4)}\}$;}\\
x_{(1,2,5,6)}, & \text{if $x=x_{(1,2,3,4)}$.}
\end{cases}
\]
Then it is easy to see that $\varphi_\alpha$ induces an embedding homomorphism  $\widetilde{\varphi}_\alpha\colon \mathbb{F}\mathcal{X}(L)\to\mathbb{F}\mathcal{X}(L')$.

\begin{lemma}\label{L:22}
Let $(\alpha', \beta')$ be a bistellar pair of type $2$ in $L$. Then $(\alpha', \beta')$ must also be a bistellar pair of ${\bm}_\alpha L$.
\end{lemma}

\begin{proof}
We note that $\Link_L\alpha'=\partial\beta'$. First we show that $\alpha'\subset \alpha\cup\beta$ is impossible. In fact, there are two possibilities: either $\alpha'\subset \alpha$, or $\alpha'\nsubseteq \alpha$. If $\alpha'\subset \alpha$, assume that $\alpha'=(1,2,3)$, then $\Link_L\alpha'$ will not be the boundary of a missing face -- so this case is impossible. If $\alpha\nsubseteq \alpha$, assume that $\alpha'=(1,2,5)$, then $\beta'=(3,4,7)$.
Since what  we have satisfies that $\Link_L\alpha'=\partial\beta'$, this induces that $(1,2,3,5,7), (1,2,4,5,7)\in L$. However,
$(1,2,3,5,7)\cap (1,2,4,5,7)=(1,2,5,7)$, which means that $(3,4)\not \in L$ since $L$ is a closed PL manifold. This is a contradiction. Therefore, $\alpha'\subset \alpha\cup\beta$ is impossible. Next, it is easy to see that $\beta'\subset \alpha\cup
\beta$ is impossible.

Now if $|\beta'\cap (\alpha\cup\beta)|\leq 2$, then clearly, whichever $|\alpha'\cap (\alpha\cup\beta)|=0$ or 1 or 2, we have that
$(\alpha', \beta')$ is a bistellar pair of ${\bf bm}_\alpha L$.
\end{proof}

By the exchange relations~(\ref{e-relation}) in Definition~\ref{seed mutation} and Lemma~\ref{L:22} we see  that all possible exchange relations in
$\mathcal{A}_{[L]}$ are also the exchange relations in $\mathcal{A}_{[{\bf bm}_\alpha L]}$. Furthermore, by Definition~\ref{alg1}, the embedding homomorphism  $\widetilde{\varphi}_\alpha\colon \mathbb{F}\mathcal{X}(L)\to\mathbb{F}\mathcal{X}({\bf bm}_\alpha L)$ induces a canonical embedding homomorphism
$e_{[L][{\bf bm}_\alpha L]}\colon \mathcal{A}_{[L]}\hookrightarrow\mathcal{A}_{[{\bf bm}_\alpha L]}$.
\end{proof}

\begin{definition}
Define
\[
\mathcal{A}_K^{PL}\,\letbe\, \lim_{[L]\in \widetilde{\mathfrak{S}}_K}\mathcal{A}_{[L]}.
\]
\end{definition}

\begin{theorem} \label{main2}
For $n=2, 4$,
assume that $K$ and $K'$ in $\mathbb{PL}_n$ are PL homeomorphic. Then $\mathcal{A}_K^{PL}\cong \mathcal{A}_{K'}^{PL}$.
\end{theorem}

\begin{proof}
This is because $K$ and $K'$ determine the same directed set.
\end{proof}

\bibliography{Final}
\bibliographystyle{plain}

\end{document}